\documentclass{amsart}

\usepackage{lineno}
\usepackage{amssymb}
\usepackage{amsmath}
\usepackage{amsthm}
\usepackage{mathrsfs}
\usepackage{bbm}
\usepackage{cite}
\usepackage{color}
\usepackage{enumitem}
\usepackage{pgfplots}

\usepackage[colorlinks,hypertexnames=false]{hyperref}
\usepackage[msc-links]{amsrefs}

\usepackage{changes}

\hypersetup{
	colorlinks=true,
	linkcolor=black,
	filecolor=black,
	urlcolor=black,
	citecolor=black,
}

\newtheorem{thm}{Theorem}[section]
\newtheorem{cor}[thm]{Corollary}
\newtheorem{defn}[thm]{Definition}
\newtheorem{lem}[thm]{Lemma}
\newtheorem{prop}[thm]{Proposition}
\newtheorem{exa}[thm]{Example}
\newtheorem{rmk}[thm]{Remark}

\numberwithin{equation}{section}

\begin{document}

\title[Extension theorem and representation formula]{Extension theorem and representation formula in non-axially symmetric domains for slice regular functions}
\author{Xinyuan Dou}
\email[Xinyuan Dou]{douxy@mail.ustc.edu.cn}
\address{Department of Mathematics, University of Science and Technology of China, Hefei 230026, China}
\author{Guangbin Ren}
\email[Guangbin Ren]{rengb@ustc.edu.cn}
\address{Department of Mathematics, University of Science and Technology of China, Hefei 230026, China}
\author{Irene Sabadini}
\email[Irene Sabadini]{irene.sabadini@polimi.it}
\address{Dipartimento di Matematica, Politecnico di Milano, Via Bonardi, 9, 20133 Milano, Italy}
\date{\today}
\thanks{This work was partly supported by the NNSF of China (11771412)}


\begin{abstract}
Slice analysis is a generalization of the theory of holomorphic functions of one complex variable to quaternions.
 Among the new phenomena which appear in this context, there is the fact that the convergence domain of $f(q)=\Sigma_{n\in\mathbb{N}}(q-p)^{*n} a_n$, given by a $\sigma$-ball $\Sigma(p,r)$, is not open in $\mathbb{H}$ unless $p\in\mathbb{R}$.  This motivates us to investigate, in this article, what is a natural topology for slice regular functions. It turns out that the natural  topology is the so-called slice topology, which  is different from the Euclidean topology and nicely adapts to the slice structure of quaternions. We extend the function theory of slice regular functions to any  domains in the slice topology.
   Many fundamental results in the classical slice analysis for axially symmetric domains fail  in our general setting.
   We can even construct a counterexample to show that a slice regular function in a domain cannot be extended  to an axially symmetric domain.
 In order to provide positive results we need to consider
 so-called path-slice functions
instead of  slice functions.
    Along this line, we can establish   an extension theorem   and a representation formula
  in   a slice-domain.
\end{abstract}

\maketitle

\section{Introduction}

The richness of complex analysis makes it natural to look for generalizations to quaternions. Around the early thirties various people, among which Moisil and Fueter,  considered possible definitions of analiticity over the quaternions. Since then, Fueter and his school started a systematic study, so the notion of `regular' quaternionic function is the one associated with the so-called Cauchy-Riemann-Fueter equation, see  \cite{Fueter1934001}
\begin{equation*}
\frac{\partial{f}}{\partial{x_1}}+i\frac{\partial{f}}{\partial{x_2}}+j\frac{\partial{f}}{\partial{x_3}}+k\frac{\partial{f}}{\partial x_4}=0.
\end{equation*}
This theory has been widely studied, see e.g. \cite{Sudbery1979001,Frenkel2015001,Libine2016001} but also \cite{CSSS,GHS} and the references therein. Unfortunately, the class of Fueter regular functions does not contain the identity function $f(q)=q$ or any other polynomial in $q$. However, Fueter \cite{Fueter1934001} found a powerful approach to construct functions in higher dimensions based on holomorphic function of one complex variable.

This approach was further developed by Sce \cite{sce}, Rinehart \cite{Rinehart1960001} and resulted  in  the theory of intrinsic or stem functions. Later on, Cullen \cite{Cullen1965001} defined another class of regular functions by intrinsic functions. Cullen regular functions contain quaternionic power series of the form $\Sigma_{n\in\mathbb{N}} q^n a_n$.

Following Cullen's approach another theory, called slice quaternionic analysis, was started by  Gentili and Struppa \cite{Gentili2006001,Gentili2007001} based on more geometric formulation. This local theory has been well established first on balls centered at the origin \cite{Gentili2006001,Gentili2007001} then over the axially symmetric slice domains \cite{Colombo2009001,Colombo2009003}. Most of the local theory of holomorphic functions of one complex variable can be lifted to quaternions.
It gives rise to the new notion of S-spectrum and has powerful applications in the quaternionic spectral theory  see e.g. \cite{ACK,Colombo2009003}, quaternionic Hilbert spaces \cite{Colombo2009003,Ghiloni2013001,Alpay2014001,Colombo2019001B}.
See \cite{Colombo2011001B,Gentili2013001B} and the references therein for other information.

In contrast to its full development in local theory, the global one remains to be developed. The challenging task of establishing the global theory over quaternions
can lead to some new theories such as slice Riemann surfaces,  domains of slice regularity, and slice Dolbeault  complexes.
Therefore, the first natural question to be answered is:

\textbf{ What is the natural topology in slice analysis? }

In \cite{Colombo2009001}, it has been argued that any slice regular function on a domain of $\mathbb H$ can be extended to an axially symmetric domain. But this is not true and we provide a counterexample in Example \ref{exa-sre}.
This means that axially symmetric slice domains are not
the maximal domains of definition of a slice regular function.
In other words,   axially symmetric domains  do  not play the  role of the natural maximal domains  in slice analysis.
On the other hand,
the convergence domain of the Taylor expansion of a slice regular function
$$\sum_{n\in\mathbb{N}}(q-p)^{*n}\frac{f^n(p)}{n!},$$
completely  described in terms of  the $\sigma$-ball $\Sigma(p,r)$ (see \cite{Gentili2012001}),
may not be a Euclidean domain. Hence the  Euclidean topology is not a natural topology in slice analysis.

To answer the above question, we observe that the slice book structure of quaternions plays a key role which makes it feasible to  lift  the  theory of holomorphic functions in  one  complex variable  to quaternions.
The slice book structure comes from the following decomposition of quaternions into complex planes,
  \begin{equation}\label{eq-ss}
	\mathbb{H}=\bigcup_{I\in\mathbb{S}}\mathbb{C}_I,
\end{equation}
where $\mathbb{C}_I=\mathbb R+I\mathbb R$ is the complex plane generated by the imaginary unit $I$ and $\mathbb S$ consists of all imaginary units $I$ of quaternions. As a result, the slice book structure of quaternions is a  natural structure in slice analysis.

Motivated by the slice book structure, we can answer  the main question of this article. It turns out that  the natural  topology  in slice analysis is the so-called slice topology, which  adapts  nicely to the book structure of quaternions.  We prove that the slice topology is finer than the Euclidean topology and all of the $\sigma$-balls  $\Sigma(p,r)$  are  domains  in the  slice topology.

With this  slice topology, some natural questions arise.  One can ask if the slice theory can be extended from the axially symmetric domains to any domains  in slice topology, but the answer is negative in general.
As an example, one can consider
the representation formula. This formula is the most important feature of the classical local theory of slice analysis. It states that  any  slice regular function  over an  axially symmetric slice domain is completely determined by its values on two pages, i.e. complex planes, of the book structure of $\mathbb H$.
This result cannot be immediately extended to the case of open sets in the slice topology. Instead, we have to extend the theory of stem functions to a new one involving  paths which  produce  path-slice functions.

Using the slice topology, one can also ask if any domain in the slice topology is a domain of holomorphy in some sense.
Also the answer to this question is negative, in general, in contrast to the case of holomorphic functions in one variable. This leads to the study of the characterization
of domains of holomorphy just like in  the case of holomorphic functions of several variables.
We provide conditions for a domain to be such a domain of holomorphy.

The structure of the paper is the following.
In Section \ref{sc-mr}, we introduce the slice topology on quaternions for slice regular functions and  we describe our main results and ideas. In Section \ref{sc-st}, we give some basic properties and examples for the slice topology. In Section \ref{sc-ip}, we prove an identity principle for slice regular functions on domains in the slice topology. In Section \ref{sc-sf}, we generalize the notion of slice function to any subset of $\mathbb{H}$ and give several equivalent definitions of slice functions. In Section \ref{sc-ef}, we prove a generalized extension formula. In Section  \ref{sc-ps}, we define a class of functions, called path-slice functions. These functions play a similar role on slice-domains as the slice functions do on axially symmetric slice domains. We also give several equivalent definitions of path-slice functions and prove our main theorem, i.e. the Representation Formula \ref{th-rf}. In Section \ref{sc-ex}, we give an example to show that the classical general representation formula \cite[Theorem 3.2]{Colombo2009001} does not work on a non-axially symmetric s-domain, using the new Representation Formula \ref{th-rf}. Section \ref{sc-dsr} is devoted to  domains of holomorphy  for slice regular functions defined on slice-open sets among which there are axially symmetric st-domains and $\sigma$-balls.

We will continue our further study on the global theory of slice analysis in some forthcoming articles.

\section{Main results}\label{sc-mr}

In this section, we state  our main results. To this end, some notation and definitions from  \cite{Gentili2006001} are needed.
   Let
\begin{equation*}
\mathbb{S}:=\{q\in\mathbb{H}:q^2=-1\}
\end{equation*}
be the  sphere of imaginary units of $\mathbb{H}$.   For any subset $\Omega$ of $\mathbb{H}$ and $I\in\mathbb{S}$, we call $$\Omega_I:=\Omega\cap\mathbb{C}_I$$ the $I$-slice (a slice) of $\Omega$.

\begin{defn}
	Assume that  $\Omega$ is  an open set in $\mathbb{C}_I$ for some  $I\in\mathbb{S}$.  A function $f:\Omega\rightarrow\mathbb{H}$ is said to be left $\mathbb{C}_I$-holomorphic (or, simply, holomorphic), if $f$ has continuous partial derivatives and satisfies
	\begin{equation}\label{eq-bp}
		\bar\partial_I f(x+yI):=\frac{1}{2}\left(\frac{\partial}{\partial x}+I \frac{\partial}{\partial y}\right) f(x+yI)=0
	\end{equation}
for any $x,y\in\mathbb{R}$ with $x+yI\in\Omega.$
\end{defn}
The definition originally given in \cite{Gentili2006001} is:
\begin{defn}\label{df-sd}
	Let $\Omega$ be a domain in $\mathbb{H}$. A function $f:\Omega\rightarrow\mathbb{H}$ is said to be (left) slice regular if $f_I:=f|_{\Omega_I}$ is left $\mathbb{C}_I$-holomorphic for any $I\in\mathbb{S}$.
\end{defn}

\cite[Theorem 8]{Gentili2012001} shows that the convergence domain of the series $$\sum_{n\in\mathbb{N}}(q-p)^{*n}a_n$$
is the $\sigma$-ball
\begin{equation*}
\Sigma(p,r):=\{q\in\mathbb{H}:\sigma(p,q)<r\}.
\end{equation*}
 with the $\sigma$-distance defined by
\begin{equation*}
\sigma(q,p):=\left\{\begin{split}
&|q-p|,&&\qquad\exists\ I\in\mathbb{S},\ s.t.\ p,q\in\mathbb{C}_I,
\\&\sqrt{(Re(q-p))^2+|Im(q)|^2+|Im(p)|^2},&&\qquad {\rm otherwise},
\end{split}\right.
\end{equation*}
for any $p,q\in\mathbb{H}$. A $\sigma$-ball is not a Euclidean domain when $p\in\mathbb{H}\backslash\mathbb{R}$. This illustrates the need to define `slice regular' functions on more sets, such as the above $\sigma$-balls. Note that the `holomorphic' condition of $f$ in Definition \ref{df-sd} is limited to each slice $\mathbb{C}_I$, $I\in\mathbb{S}$. Thus in order to define `slice regularity', we just need to guarantee that $\Omega_I$ is open in $\mathbb{C}_I$ for each $I\in\mathbb{S}$.

\begin{defn}\label{df-so}
	A subset $\Omega$ of $\mathbb{H}$ is called slice-open, if $\Omega_I$ is open in $\mathbb{C}_I$ for any $I\in\mathbb{S}$.
\end{defn}

It is clear that the $\sigma$-ball $\Sigma(p,r)$ is slice-open. Now we extend Definition \ref{df-sd} to slice-open sets.

\begin{defn}
	Let $\Omega$ be a slice-open set in $\mathbb{H}$. A function $f:\Omega\rightarrow\mathbb{H}$ is called (left) slice regular, if $f_I$ is left holomorphic for any $I\in\mathbb{S}$.
\end{defn}

We note that, so far, in the literature, numerous results in
 slice quaternionic analysis (according to the definition in \cite{Gentili2006001})  have been developed systematically over axially symmetric slice domains and this is basically enough for various purposes. Our goal is to generalize it to any slice-open set. Some properties can be proved  as in the classical case, e.g. the following Splitting Lemma. Thus we state it without proof.

\begin{lem}\label{lm-sl}
	(Splitting Lemma)
	Let $f$ be a function on a slice-open set $\Omega$. Then $f$ is slice regular, if and only if for all $I,J\in\mathbb{S}$ with $I\bot J$, there are two complex-valued holomorphic functions $F,G: \Omega_I\rightarrow\mathbb{C}_I$ such that $f_I=F+GJ$.
\end{lem}

The set of slice-open sets gives a topology on $\mathbb{H}$, in fact we have:

\begin{lem}\label{lm-sl123}
 $$\tau_s(\mathbb{H}):=\{\Omega\subset\mathbb{H}:\Omega\ \mbox{is slice-open}\}$$ is a topology of $\mathbb{H}$.
\end{lem}

\begin{proof}
 { The claim can be immediately verified with a direct proof or by observing that the slice topology is the final topology with respect to the inclusions $\{i_I:\mathbb{C}_I\rightarrow\mathbb{H}\}_{I\in\mathbb{S}}$.}
\end{proof}
	
\begin{defn}\label{df-sth}
	We call $\tau_s(\mathbb{H})$ the slice topology. Open sets, connected sets and paths in the slice topology are called  slice-open sets, slice-connected and slice-paths.
\end{defn}

\begin{rmk} In particular, a similar terminology will be used for all the other notions in the slice topology, with one remarkable exception. We will not use the terminology slice-domain to denote a domain in the slice topology, since this notion is already used in the literature to denote something different (see Definition \ref{df-sdr} below). We will use instead the term {\em slice topology-domain}, in short, {\em st-domain}.
\end{rmk}

\begin{defn}\label{df-sdr}
A set  $\Omega$ in $\mathbb{H}$ is called classical slice domain, in short s-domain,   if  $\Omega$ is a domain in the Euclidean topology,
 $$\Omega_\mathbb{R}:=\Omega\cap \mathbb R\neq\varnothing, $$ and $\Omega_I$  is a domain in $\mathbb{C}_I$ for any $I\in\mathbb{S}$.
\end{defn}

It is evident that an $s$-domain must be a domain in the slice topology, i.e. an st-domain, but the converse statement is not true  (see Example \ref{ex-sd}).

The classical slice quaternionic analysis is established on axially symmetric s-domains.
The slice quaternionic analysis on st-domains shows differences with respect to the classical one, since it relies on the slice-connectedness. For example, the proof of the following generalized Identity Principle in Section \ref{sc-ip}, involves some properties of st-domains induced by slice-connectedness.

\begin{thm}\label{th-ip}
	(Identity Principle)
	Let  $f$ and $g$ be two slice regular functions on an st-domain $\Omega$ in $\mathbb{H}$.  If  $f$ and $g$ coincide on a subset of $\Omega_I$ with an accumulation point in $\Omega_I$ for some $I\in\mathbb{S}$, then $f=g$ on $\Omega$.
\end{thm}

Another fundamental result in the classical slice analysis is the general representation formula \cite[Theorem 3.2]{Colombo2009001}.
Unfortunately, this formula fails, in general, on non axially symmetric domains, see Section \ref{sc-ex}.

To get the validity of the formula, we have to introduce the notion of path-slice functions, see Definition \ref{df-ps}.

 We consider the transform
 \begin{align*}
	\mathcal{P}_I:\mathbb{C}&\rightarrow\mathbb{C}_I
\\
 x+yi&\mapsto x+yI
	\end{align*}
for any $x,y\in\mathbb{R}\ \mbox{and}\ I\in\mathbb{S}$.  For any path $\gamma$ in $\mathbb{C}$, we define its corresponding path in $\mathbb C_I$ as
 $$\gamma^I:=\mathcal{P}_I\circ\gamma$$ for any $I\in\mathbb S$.

\begin{thm}(Representation Formula)\label{th-rf}
	Assume that  $\Omega$ is a slice-open set  in $\mathbb{H}$ and
suppose  $\gamma$ is a path in $\mathbb{C}$ satisfying the conditions
$$ \gamma(0)\in\mathbb{R}, \qquad \gamma^I,\gamma^J,\gamma^K\subset\Omega$$
for some $ I,J,K\in\mathbb{S}$ with  $J\neq K$. If $f$ is  a slice regular function  on   $\Omega$, then
	\begin{equation}\label{eq-rf}
	f\circ\gamma^I=(I-K)(J-K)^{-1}f\circ\gamma^J+(I-J)(K-J)^{-1}f\circ\gamma^K.
	\end{equation}

\end{thm}

 \begin{rmk} Although we only assume the domain $\Omega$ in consideration is slice-open, some restrictions related to slice-connectedness are implicitly  involved
 as shown by the conditions
 $$\gamma^I,\gamma^J,\gamma^K\subset\Omega.$$
   The path $\gamma^I$ in a slice can distinguish points of $\Omega$ more finely than $x+yI$ (by the Euclidean coordinate in $\mathbb{C}_I$), see Section \ref{sc-ex}. This ensures that the representation formula holds on non-axially symmetric domains.
\end{rmk}

A function satisfying \eqref{eq-rf} is called path-slice in Section \ref{sc-ps} based on   an equivalent definition.
It turns out that any slice regular function is a   path-slice   function.
 The proof of  \eqref{eq-rf}
 shall depend on a new approach; see Proposition \ref{pr-ps} (i) and (vi).

\section{Slice topology}\label{sc-st}
  In this section, we study some properties of the slice topology $\tau_s(\mathbb H)$. The slice structure induces the intricacy of the notion of slice-connectedness near the real axis. We tackle this issue in terms of  slice-paths.

We denote by $\tau_s(\mathbb H)$ and $\tau(\mathbb H)$ the  slice topology  and the Euclidean topology of $\mathbb{H}$, respectively.
Sometimes, we simply write  $\tau_s$ and $\tau$, for short.

\begin{prop}
	$(\mathbb{H},\tau_s)$ is a Hausdorff space and $\tau\subsetneq\tau_s$.
\end{prop}

\begin{proof}
	Since every Euclidean open set in $\mathbb{H}$ is slice-open, we have $\tau\subset\tau_s$ and $\tau_s$ is Hausdorff. Note that  $\Sigma(p,r)$ is slice-open and not open for any $p\in\mathbb{H}\backslash\mathbb{R}$ and $r\in\mathbb{R}_+$. It follows that the slice topology is strictly finer than the Euclidean topology.
\end{proof}

We remark that the slice topology locally coincides with the Euclidean topology on a slice complex plane for any point away from the real axis $\mathbb{R}$, because for any $I\in\mathbb{S}$ the subspace topologies on $\mathbb{C}_I$ of $\tau_s(\mathbb H)$ and $\tau(\mathbb H)$ coincide, i.e.
 $$ \tau_s(\mathbb{C}_I)=\tau(\mathbb{C}_I).$$
However,  $\tau_s(\mathbb H)$ is quite different from the Euclidean topology $\tau(\mathbb H)$ near $\mathbb{R}$ as  demonstrated by the following example.

\begin{exa}\label{ex-nsi}
	Fix $I\in\mathbb{S}$. We construct a slice-open set $\Omega$ in $\mathbb{H}$ as
\begin{equation}\label{eq-ob}
	\Omega:=\bigcup_{J\in\mathbb{S}}\Omega_J,
	\end{equation}
where
	\begin{equation*}
	\Omega_J:=\left\{
	\begin{aligned}
	&\{x+yJ\in\mathbb{C}_J:x^2+\frac{y^2}{\mbox{dist}(J,\mathbb{C}_I)}<1\},\qquad&&J\neq\pm I,
	\\&\{x+yJ\in\mathbb{C}_J:x^2+y^2<1\}, &&J=\pm I.
	\end{aligned}\right.
	\end{equation*}
	Here $\mbox{dist}(J,\mathbb{C}_I)$ is the Euclidean distance from $J$ to $\mathbb{C}_I$.

By the construction, we know that $\Omega$ is slice-open. But $\Omega$ is not  open in $\mathbb{H}$ since $0\in\Omega$ and  $0$ is
     not  in the  Euclidean interior  of $\Omega$.
     This is because $\Omega_J$ is an ellipse  whose
      minor semi-axis $\sqrt{\mbox{dist}(J,\mathbb{C}_I)}$  tends to $0$, when $J$ approaches $I$ with $J\neq\pm I$.
\end{exa}

The slice topology is finer than the topology induced by the $\sigma$-distance as proved in the following result:
	\begin{prop}
		$\tau_\sigma\subsetneq\tau_s$, where $\tau_\sigma$ is the topology on $\mathbb{H}$ induced by the $\sigma$-distance.
	\end{prop}

\begin{proof}
	 Let $U\in\tau_\sigma$ and $I\in\mathbb{S}$. Then for any $q\in U_I$,
	\begin{equation*}
		r:=\sigma(q,\mathbb{H}\backslash U)>0.
	\end{equation*}
	Note that for each $z,w\in\mathbb{C}_I$, $\sigma(z,w)=dist_{\mathbb{C}_I}(z,w)$, where $dist_{\mathbb{C}_I}(z,w)$ is the Euclidean distance in $\mathbb{C}_I$. Let $B_I(q,r)$ be the ball with center $q$ and radius $r$ in $\mathbb{C}_I$. It is clear that $B_I(q,r)$ is a subset of $U_I$, and $q$ is a point in the interior of $U_I$. Hence $U_I$ is open in $\mathbb{C}_I$ so that $U$ is slice-open and $\tau_\sigma\subset\tau_s$.
	
	To show that the slice topology is strictly finer, we consider the set $\Omega$ defined in \eqref{eq-ob}, Example \ref{ex-nsi}, which is a slice-open set. Let $J\in\mathbb{S}$. Since $\mathbb{H}\backslash\Omega\supset\mathbb{C}_J\backslash\Omega_J$,
	\begin{equation}\label{eq-s0}
		\sigma(0,\mathbb{H}\backslash\Omega)\le \sigma(0,\mathbb{C}_J\backslash\Omega_J)= dist_{\mathbb{C}_J}(0,\mathbb{C}_J\backslash\Omega_J).
	\end{equation}
	Note that
	\begin{equation}\label{eq-lj}
		\lim_{J\rightarrow I, J\neq I}\left[dist_{\mathbb{C}_J}(0,\mathbb{C}_J\backslash\Omega_J)\right]=0.
	\end{equation}
From \eqref{eq-s0} and \eqref{eq-lj} we deduce that $\sigma(0,\mathbb{H}\backslash\Omega)=0$. Hence, $0$ is not an interior point in $\Omega$ under the topology $\tau_\sigma$ and so $\Omega$ is not open in $\tau_\sigma$. However, $\Omega$ is a slice-open set and we conclude that $\tau_\sigma\neq\tau_s$.
\end{proof}

To deal with the difficulties of the topology   near $\mathbb{R}$,   a new notion, called real-connectedness, comes up. This provides an effective tool since the slice topology has a real-connected subbase.

\begin{defn}
	A subset $\Omega$ of $\mathbb{H}$ is called real-connected, if $$\Omega_{\mathbb{R}}:=\Omega\cap\mathbb{R}$$ is connected in $\mathbb{R}$. In particular, when $\Omega\cap\mathbb{R}=\varnothing$, $\Omega$ is real-connected.
\end{defn}

\begin{prop}\label{pr-sor}
	For any slice-open set $\Omega$ in $\mathbb{H}$ and $q\in\Omega$, there is a real-connected st-domain $U\subset\Omega$ containing $q$.
\end{prop}

\begin{proof} We take $U$  to be the slice-connected component of the set
$$(\Omega\backslash\Omega_{\mathbb{R}})\cup A$$ containing $q$.
 Here when
    $q\in\mathbb{R}$,  we take $A$ to be the connected component of $\Omega_{\mathbb{R}}$ containing $q$ in $\mathbb{R}$;
  otherwise, we set $A:=\varnothing$.

  It is easy to check that $q\in U$ and $U$ is a real-connected st-domain.
\end{proof}

Now we describe slice-connectedness  by means of  slice-paths.

\begin{defn}
	A path $\gamma$ in $(\mathbb{H},\tau)$ is said to be on a slice, if $\gamma\subset\mathbb{C}_I$ for some $I\in\mathbb{S}$.
\end{defn}

\begin{prop}
	Every path on a slice is a slice-path.
\end{prop}

\begin{proof}
	It follows directly from the fact that $\tau_s(\mathbb{C}_I)=\tau(\mathbb{C}_I)$ for any $I\in\mathbb{S}$.
\end{proof}

\begin{prop}\label{pr-rcsd}
	Assume that an st-domain $U$ is real-connected.
	\begin{enumerate}[label=(\roman*)]
		
		\item If $U_\mathbb{R}=\varnothing$, then $U\subset\mathbb{C}_I$ for some $I\in\mathbb{S}$.
		
		\item If $U_\mathbb{R}\neq\varnothing$, then for any $q\in U$ and $x\in U_\mathbb{R}$, there exists  a path on a slice from $q$ to $x$.
			
	\end{enumerate}
\end{prop}

\begin{proof}
	(i)  If $U_\mathbb{R}=\varnothing$, then
	\begin{equation*}
	U\subset\bigsqcup_{J\in\mathbb{S}}\mathbb{C}^+_J,
	\end{equation*}
	where $$\mathbb{C}_J^+:=\{x+yJ\in\mathbb{H}:y>0\}$$
 is a slice-open set in $\mathbb{H}$ for any $J\in\mathbb{S}$.
This means that
$$U\subset\mathbb{C}_I^+$$ for some $I\in\mathbb{S}$ since
$U$ is  slice-connected,

	(ii) We fix $q\in U$ and $x\in U_\mathbb{R}$. Take $I\in\mathbb S$ such that  $q\in\mathbb{C}_I$.
 Since $U$ is an st-domain in $\mathbb H$,  by definition $U_I$ is an open set in the plane $\mathbb C_I$.
 Let $V$ be the
  connected component of $U_I$ containing $q$.

By definition we have that $\mathbb{C}_I\backslash\mathbb{R}$ and $\bigcup_{J\in\mathbb{S}\backslash\{\pm I\}}(\mathbb{C}_J\backslash\mathbb{R})$ are slice-open.
  If $V_\mathbb{R}=\varnothing$, then

  \begin{equation*}
  V=U\cap(\mathbb{C}_I\backslash\mathbb{R})\qquad\mbox{and}\qquad U\backslash V=U\cap\bigg[\bigcup_{J\in\mathbb{S}\backslash\{\pm I\}}(\mathbb{C}_J\backslash\mathbb{R})\bigg]
  \end{equation*}
 are slice-open . Since $U$ is slice-connected and nonempty, it follows from
  $$  U= V \bigsqcup (U\backslash V)$$
    that $V=U$.  This implies  $U_\mathbb{R}=V_\mathbb{R}=\varnothing$, which is a contradiction. We thus conclude
    $$V_\mathbb{R}\neq\varnothing. $$

We take a point  $x_0\in V_\mathbb{R}$.  Since
$V$ is  the
  connected component of $U_I$ containing $q$,
there exists a path $\alpha$ in $V$ from $q$ to $x_0$.
Because $U$ is real-connected, we have  a path $\beta$ in $U_\mathbb{R}$ from $x_0$ to $x$. It is clear that $\alpha\beta$ is a path on a slice from $q$ to $x$.
	\end{proof}

\begin{cor}\label{pr-rcsd394}
	Assume that an st-domain $U$ is real-connected.
	\begin{enumerate}[label=(\roman*)]

		\item $U_I$ is a domain in $\mathbb{C}_I$ for any $I\in\mathbb{S}$.
		
		\item For any $p,q\in U$, there exist two paths $\gamma_1,\gamma_2$ such that each of them is a path on a slice  in $U$,   $\gamma_1(1)=\gamma_2(0)$,  and $\gamma_1\gamma_2$ is a slice-path from $p$ to $q$.
		
	\end{enumerate}
\end{cor}

\begin{proof}
This follows directly from Proposition \ref{pr-rcsd}.
\end{proof}

\begin{prop}
	The topological space $(\mathbb{H},\tau_s)$ is connected, locally path-connected and path-connected.
\end{prop}

\begin{proof}
	It  follows from Proposition \ref{pr-sor} and Corollary  \ref{pr-rcsd394} (ii) that  $(\mathbb{H},\tau_s)$ is locally path-connected.  Since $\mathbb{H}\cap\mathbb{C}_I=\mathbb{C}_I\supset\mathbb{R}$ for any $I\in\mathbb{S}$, we have $(\mathbb{H},\tau_s)$ is path-connected so  that it is also  connected.
\end{proof}

\begin{cor}\label{co-dsd}
	A set $\Omega\subset\mathbb{H}$ is an st-domain if $\Omega_\mathbb{R}\neq\varnothing$ and $\Omega_I$ is a domain in $\mathbb{C}_I$ for any $I\in\mathbb{S}$.
\end{cor}

\begin{proof}
If $\Omega_I$ is open for any $I\in\mathbb{S}$, then by definition $\Omega$ is slice-open.
Since  $\Omega_{\mathbb{R}}\neq\varnothing$, we can take a fixed point  $x\in\Omega_{\mathbb{R}}$. By hypothesis, $\Omega_I$ is a domain in $\mathbb{C}_I$ for any $I\in\mathbb{S}$, there is a path on a slice from $x$ to each point of $\Omega$. It implies that $\Omega$ is slice-path-connected so that it is also  slice-connected. Thus $\Omega$ is an st-domain.
\end{proof}
Note that there are sets $\Omega$ which are st-domains and such that $\Omega_\mathbb{R}=\varnothing$.
For example, let us consider a fixed $J\in\mathbb{S}$. We set
\begin{equation*}
	\Omega:=B_J(2J,1)=\{q\in\mathbb{C}_J:|q-2J|<1\}.
\end{equation*}
It is evident that $\Omega$ is an st-domain and $\Omega_\mathbb{R}=\varnothing$.

\begin{rmk}\label{rmk-sd}
	By Corollary \ref{co-dsd}, any s-domain is an st-domain. Therefore  the notion of st-domain is a generalization of the notion of s-domain.
\end{rmk}

However not every st-domain $\Omega$ is an s-domain, even when $\Omega$ is a domain in $\mathbb{H}$, as we show in the following example.

\begin{exa}\label{ex-sd} We fix $I\in\mathbb{S}$ and consider a   domain in $\mathbb{H}$, defined by
 $$\Omega:=B(0,2)\cup B(6,2)\cup U,$$
where
	\begin{eqnarray*} U&:=& \{q\in\mathbb{H}:\mbox{dist}(q-I,[0,6])<\frac{1}{2}\}.\end{eqnarray*}

It is easy to check that
$$\Omega_J=B_J(0,2)\cup B_J(6,2)$$
 for any $J\in\mathbb{H}$ with $J\bot I$.
Hence  $\Omega_J$ is not connected in $\mathbb{C}_J$ so that  $\Omega$ is not an s-domain. However $\Omega$ is slice-connected, because any point in $\Omega$ can be connected to $0$ or $6$ by a path in a slice, and $0$ can be connected to $6$ by a path in $\mathbb{C}_I$. And since $\Omega_J$ is open in $\mathbb{C}_J$ for any $J\in\mathbb{S}$, $\Omega$ is an st-domain.
\end{exa}

\section{Identity Principle}\label{sc-ip}

In this section we provide an identity principle for slice regular functions defined on st-domains.

Since the st-domains satisfy  conditions weaker than those one required by s-domains,  the proof of  the  identity principle \ref{th-ip} is more difficult than the one for s-domains.
We need to reduce the problem to the special case where the domain is  real-connected.

\begin{lem}\label{lm-ip}
Assume that an st-domain  $\Omega$  is real-connected.  Let $f$ and $g$ be two slice regular functions on $\Omega$. If $f$ and $g$ coincide on a subset of $\Omega_I$ with an accumulation point in $\Omega_I$ for some $I\in\mathbb{S}$, then $f=g$ on $\Omega$.
\end{lem}

\begin{proof}
By assumption, we have $\Omega_I\neq\varnothing$ so that Corollary  \ref{pr-rcsd394} (i) implies $\Omega_I$ is a non-empty domain in $\mathbb{C}_I$. Therefore, using the Splitting Lemma and the identity principle for classical holomorphic functions of a complex variable, we deduce that $f$ and $g$ coincide on $\Omega_I$.

If $\Omega_{\mathbb{R}}=\varnothing$, then $\Omega=\Omega_I$ due to  Proposition \ref{pr-rcsd} (i) so that  $f=g$ on $\Omega$.

 Otherwise, we have $\Omega_{\mathbb{R}}\neq\varnothing$. By Corollary \ref{pr-rcsd394} (i), $\Omega_J$ is a domain in $\mathbb{C}_J$ for all $J\in\mathbb{S}$. Since $f=g$ on $\Omega_\mathbb{R}(\subset\Omega_I)$, it follows that $f=g$ on $\Omega_J$ for any $J\in\mathbb{S}$.
	Consequently, $f=g$ on $\Omega=\bigcup_{J\in\mathbb{S}}\Omega_J$.
\end{proof}

Now we can give the proof of the identity principle for st-domains.

\begin{proof}[Proof of Theorem \ref{th-ip}]
	We consider the set
	\begin{equation*}
	A:=\{x\in\Omega:\exists\ V\in\tau_s(\Omega),\ \mbox{s.t.}\ x\in V\ \mbox{and}\ f=g\ \mbox{on}\ V\}.
	\end{equation*}
By definition,  $A$ is a slice-open set in $\Omega$.
	
	Next, we come to show that  $A$ is nonempty.  Due to Proposition \ref{pr-sor}, there exists  a real-connected st-domain $U$ such that it  contains  the accumulation point $p$ and  $U\subset\Omega$. It follows  from  Lemma \ref{lm-ip} applied to $U$ that  $f=g$ on $U$. This means  that $p\in A$  so that  $A$ is nonempty.

Finally, we claim that  $\Omega\backslash A$ is slice-open. From this claim and the fact that $\Omega$ is slice-connected, we conclude that $A=\Omega$ so that $f=g$ on $\Omega$.

It remains to prove the claim.
	Let  $q\in\Omega\backslash A$ be arbitrary. From Proposition \ref{pr-sor}, there exists  a real-connected st-domain $V$ containing $q$ with $V\subset\Omega$. We already know that both $A$ and $V$ are slice-open,  so is $A\cap V$.

If $A\cap V \neq\varnothing$, then $A\cap V$ is a non-empty slice-open.
 Since $f=g$ on $A\cap V$, it follows from  Lemma \ref{lm-ip} that  $f=g$ on $V$. This means that  $q\in A$, a contradiction.

Therefore, we have  $$A\cap V=\varnothing.$$
This implies that  $q$ is a slice-interior  point of $\Omega\backslash A$. Hence  $\Omega\backslash A$ is slice-open. This proves the claim and finishes the proof.
\end{proof}

\section{Slice Functions}\label{sc-sf}

Slice functions play a fundamental role in the theory of slice regular functions. The related stem function theory for slice analysis has been established in the case of real alternative $*$-algebras  \cite{Ghiloni2011001}. See  \cite{Jin2020001} for a recent development.

In this section, we give several equivalent characterizations of slice functions. For convenience, we consider slice functions on an arbitrary domain of definition.

We remark that our definition of the slice function is a different form of the classical one.
\begin{defn}\label{df-s} Let  $\Omega$ be an arbitrary set in $\mathbb{H}$.
	A function  $f:\Omega\rightarrow\mathbb{H}$    is called a slice function if there is a function $F:\mathbb{R}^2\rightarrow\mathbb{H}^{2\times 1}$  such that
	\begin{equation}\label{eq-sf}
	f(x+yI)=(1,I)F(x,y)
	\end{equation}
	for any  $x+yI\in\Omega$ such that $x,y\in\mathbb{R}$,  $I\in\mathbb{S}$, and  $y\ge 0$.

The function $F$ is  referred to as an upper stem function of the slice function $f$.
\end{defn}
We note that we are not requiring, at this stage, any condition on $F$ and since it is defined in $\mathbb{R}^2$,  for $x+Iy\not\in\Omega$, we set $F(x,y)=(0, 0)^T$.
Let us denote
\begin{equation*}
\mathbb{S}^2_*:=\{(I,J)\in\mathbb{S}^2:I\neq J\}.
\end{equation*}
For any $(J,K)\in\mathbb{S}^2_*$ we have the noteworthy identity
\begin{equation}\label{eq-jk}
(J-K)^{-1}J=-K(J-K)^{-1}.
\end{equation}
From this, it is easy to check that
\begin{equation}\label{eq-im}
\left(\begin{matrix}
1&J\\1&K
\end{matrix}\right)^{-1}
=\left(\begin{matrix}
(J-K)^{-1}J&(K-J)^{-1}K\\(J-K)^{-1}&(K-J)^{-1}
\end{matrix}\right).
\end{equation}

\begin{prop}\label{pr-sfe} For any  function $f:\Omega\rightarrow\mathbb{H}$ with $\Omega\subset\mathbb H$, the following statements are equivalent:
	\begin{enumerate} [label=(\roman*)]
	
		\item The function $f$ is  a slice function.
		
		\item  There exists a function  $F:\mathbb{R}^2\rightarrow\mathbb{H}^{2\times 1}$ such that
		\begin{equation}\label{eq-stf}
		f(x+yI)=(1,I)F(x,y)
		\end{equation}
		for any  $x+yI\in\Omega$ with  $x,y\in\mathbb{R}$ and any $I\in\mathbb{S}$.
		
		\item If    $x,y\in\mathbb{R}$, $I\in\mathbb S$,   and  $(J, K)\in\mathbb S^2_*$  such that
	 	$x+yL\in\Omega$  for $L=I, J, K$,   then
		\begin{equation}\label{eq-mf}
		f(x+yI)=(1,I)\left(\begin{matrix}
		1&J\\1&K
		\end{matrix}\right)^{-1}
		\left(\begin{matrix}
		f(x+yJ)\\f(x+yK)
		\end{matrix}\right)
		\end{equation}

		\item If    $x,y\in\mathbb{R}$, $I\in\mathbb S$,   and  $(J, K)\in\mathbb S^2_*$  such that
	 	$x+yL\in\Omega$  for $L=I, J, K$,   then
		\begin{equation}\label{eq-il}
		\begin{split}
		f(x+yI)=&(J-K)^{-1}[Jf(x+yJ)-Kf(x+yK)]\\&+I(J-K)^{-1}[f(x+yJ)-f(x+yK)].
		\end{split}
		\end{equation}
		
		\item  If    $x,y\in\mathbb{R}$, $I\in\mathbb S$,   and  $(J, K)\in\mathbb S^2_*$  such that
	 	$x+yL\in\Omega$  for $L=I, J, K$,   then
		\begin{equation}\label{eq-lf}
		f(x+yI)=(I-K)(J-K)^{-1}f(x+yJ)+(I-J)(K-J)^{-1}f(x+yK).
		\end{equation}
		
	\end{enumerate}
\end{prop}

\begin{proof}  It follows from \eqref{eq-jk} and \eqref{eq-im}
	  that assertions $(iii)$, $(iv)$, $(v)$ are equivalent.
	
	(i)$\Rightarrow\ (ii)$. If $f$ is a slice function, then there is a function $G=(G_1,G_2)^T:\mathbb{R}^2\rightarrow\mathbb{H}^{2\times 1}$ such that
	\begin{equation*}
	f(x+yI)=(1,I)G(x,y)
	\end{equation*}
	for any  $x+yI\in\Omega$ with $x,y\in\mathbb{R}$,  $I\in\mathbb{S}$,  and  $y\ge 0$.

Hence we can take function  $F:\mathbb{R}^2 \rightarrow\mathbb{H}^{2\times 1}$  defined by
	\begin{equation*}
	F(x,y):=\left\{\begin{split}
		&(G_1,G_2)^T(x,y),&\qquad y\ge 0,
		\\&(G_1,-G_2)^T(x,-y),&\qquad y<0.
		\end{split}\right.
	\end{equation*}
Direct calculation shows that  \eqref{eq-stf} holds.
	
	$(ii)\Rightarrow\ (iii)$. According to \eqref{eq-stf}, we have
	\begin{equation*}
	\left(\begin{matrix}
	f(x+yJ)\\f(x+yK)
	\end{matrix}\right)
	=\left(\begin{matrix}
	1&J\\1&K
	\end{matrix}\right)F(x,y)
	\end{equation*}
	for any $x,y\in\mathbb{R}$ and $(J,K)\in\mathbb{S}^2_*$. This implies that
	\begin{equation}\label{eq-ma}
	F(x,y)=\left(\begin{matrix}
	1&J\\1&K
	\end{matrix}\right)^{-1}
	\left(\begin{matrix}
	f(x+yJ)\\f(x+yK)
	\end{matrix}\right).
	\end{equation}
	Combining this with \eqref{eq-stf}, we deduce that \eqref{eq-mf} holds.
	
	$(iii)\Rightarrow\ (i)$. We consider the sets
	\begin{equation*}
	\mathcal{A}:=\{(x,y)\in\mathbb{R}^2:y\ge 0\ \mbox{and}\ |(x+y\mathbb{S})\cap\Omega|=1\}
	\end{equation*}
	and
	\begin{equation*}
	\mathcal{B}:=\{(x,y)\in\mathbb{R}^2:y\ge 0\ \mbox{and}\ |(x+y\mathbb{S})\cap\Omega|>1\},
	\end{equation*}
where we denoted by $|S|$ the cardinality of the set $S$.

If   $(x,y)\in\mathcal{B}$, then there are at least two distinct  points in the set $(x+y\mathbb{S})\cap\Omega$.
Therefore, 	the axiom of choice shows
that we can choose    $(J_{x,y},K_{x,y})\in\mathbb{S}^2_*$   such that
$$x+yJ_{x,y}, \qquad  x+yK_{x,y}\in  (x+y\mathbb{S})\cap\Omega$$
  for any $(x,y)\in\mathcal{B}$.

  From this, we can construct a function  $G:\mathcal{B}\rightarrow\mathbb{H}^{2\times 1}$ defined    by
	\begin{equation*}
	G(x,y):=\left(\begin{matrix}
	1&J_{x,y}\\1&K_{x,y}
	\end{matrix}\right)^{-1}
	\left(\begin{matrix}
	f(x+yJ_{x,y})\\f(x+yK_{x,y})
	\end{matrix}\right).
	\end{equation*}

Finally, 	we can define our desired  function $F:\mathbb{R}^2\rightarrow\mathbb{H}^{2\times 1}$ via
	\begin{equation*}
	F(x,y):=\left\{\begin{split}(&f(x+yI_{x,y}),0)^T,&\qquad (x,y)\in\mathcal{A},
	\\&G(x,y),&\qquad (x,y)\in\mathcal{B},\\&(0, 0)^T,&\qquad\mbox{otherwise},
	\end{split}
	\right.
	\end{equation*}
	where $I_{x,y}$ is the unique imaginary unit $I\in\mathbb{S}$ such that $x+yI\in\Omega$ for $(x,y)\in\mathcal{A}$.
	It is  easy to check that $F$ satisfies \eqref{eq-sf} so that  $f$ is a slice function.
\end{proof}

We remark that the form in \eqref{eq-lf} is also in \cite[Proposition 6]{Ghiloni2011001} for the related class of functions.

\begin{rmk}\label{rm-rf}
By Proposition \ref{pr-sfe}, the classical representation formula in \cite[Theorem 3.2]{Colombo2009001} can be interpreted in the formalism of slice functions. That is,  any slice regular function defined on an axially symmetric s-domain is a slice function.
\end{rmk}

\section{Extension Theorem}\label{sc-ef}

In \cite[Theorem 4.2]{Colombo2009001}, the extension theorem is generalized from balls centered on the real axis to axially symmetric s-domains.
In this section, we consider its further generalization to non necessarily axially symmetric st-domains.

For any $\mathbb I=(I_1,I_2)\in\mathbb{S}^2_*$, we set
\begin{equation*}
{\tau}[\mathbb I]:=\{(U,V):U\in\tau(\mathbb{C}_{I_1})\ \mbox{and}\ V\in\tau(\mathbb{C}_{I_2})\}.
\end{equation*}

Associated with $$\mathbb U=(U_1,U_2)\in{\tau}[\mathbb I], \quad \mbox{with}
 \quad \mathbb I=(I_1,I_2)\in\mathbb{S}^2_*, $$  we introduce the following three sets:
\begin{eqnarray*}
\mathbb U^+_{s} &:=& \left(U_1\cap \mathbb{C}_{I_1}^+\right) \bigsqcup \left(U_2\cap\mathbb{C}_{I_2}^+\right)\bigsqcup \left(U_1\cap U_2\cap\mathbb{R}\right),
\\
\mathbb U^\Delta_s &:=& \{x+y \mathbb{S}:   (x+yI_1, x+yI_2) \in\mathbb U,    y\in\mathbb{R},  \ y\ge 0\},
\\
\mathbb U^{+\Delta}_s &:=& \mathbb U^+_{s}\cup\mathbb U^\Delta_s.
\end{eqnarray*}
Sometimes we also replace  $\mathbb U^{+\Delta}_s$ by $\mathbb U^{+\Delta}_{s,\mathbb I}$ to emphasize its dependence on $\mathbb I$.

\begin{lem} Let $\mathbb I\in\mathbb{S}^2_*$ be fixed, then
	$\mathbb U^{+\Delta}_s$ is slice-open.
\end{lem}

\begin{proof} We need to show that
  any $q\in\mathbb U^{+\Delta}_s$ is a slice-interior point of $\mathbb U^{+\Delta}_s$.

Case 1:  $q\in\mathbb U^{+\Delta}_s\backslash\mathbb{R}$:

If $q\in\mathbb U^+_{s}\backslash\mathbb{R}$, then $q$ is an interior point of $U_1\cap\mathbb{C}_{I_1}^+$ or $U_2\cap\mathbb{C}_{I_2}^+$.
Hence $q$ is a slice-interior point of $\mathbb U^{+}_s$ as well as $\mathbb U^{+\Delta}_s$.

If  $q\in\mathbb U^\Delta_s\backslash\mathbb{R}$, it can be expressed as  $$q=x+yJ$$ for some $J\in\mathbb{S}$, $x,y\in\mathbb{R}$ with $y>0$. By definition of $\mathbb U_s^\Delta$, $$x+yI_1\in U_1\cap\mathbb{C}_{I_1}^+, \qquad  x+yI_2\in  U_2\cap\mathbb{C}_{I_2}^+.$$
Hence there exists  an $r\in\mathbb{R}_+$ such that
$$B_{I_1}(x+yI_1,r)\subset U_1\cap\mathbb{C}_{I_1}^+, \qquad B_{I_2}(x+yI_2,r)\subset U_2\cap\mathbb{C}_{I_2}^+.$$
This means  $$B_J(x+yJ,r)\subset\mathbb U^\Delta_s$$ so that $q$ is a slice-interior of $\mathbb U^{+\Delta}_s$.

\bigskip

Case 2:  $q\in\mathbb U^{+\Delta}_s\cap\mathbb{R}$: 	

	It is easy to check that $$\mathbb U^{+\Delta}_s\cap\mathbb{R}=U_1\cap U_2\cap\mathbb{R}.$$
Since  $q\in\mathbb U^{+\Delta}_s\cap\mathbb{R}$, there exists an  $r\in\mathbb{R}_+$ such that
 $$B_{I_1}(q,r)\subset U_1, \qquad B_{I_2}(q,r)\subset U_2, $$
which  implies, by definition, that
 $$B_J(q,r)\subset\mathbb U^\Delta_s$$ for any $J\in\mathbb{S}$. Hence $B(q,r)\subset\mathbb U^{+\Delta}_s$.
\end{proof}

\begin{thm}\label{pr-ef} Let
  $\mathbb I\in\mathbb{S}^2_*$  and  $\mathbb U=(U_1, U_2)\in{\tau}[\mathbb I]$.  If  $f:U_1\cup U_2\rightarrow\mathbb{H}$ is a function such that  $f|_{U_1}$ and $f|_{U_2}$
  are both holomorphic, then the function  $f|_{\mathbb U^+_{s}}$ admits
 a slice regular extension $\widetilde{f}$ to $\mathbb U^{+\Delta}_s$.
	
	Moreover, if $W$ is an st-domain such that
$$W\subseteq\mathbb U^{+\Delta}_s, \qquad W\cap\mathbb U^+_{s}\neq\varnothing, $$  then $\widetilde{f}|_W$ is a slice function and it is the unique slice regular extension on $W$  of $f|_{W\cap\mathbb U^+_{s}}$.
\end{thm}

\begin{proof}	
	Define a function $g:\mathbb U^\Delta_s \longrightarrow \mathbb H$ by
	\begin{equation}\label{eq-fji}
	g(x+yJ):=(J-I_2)(I_1-I_2)^{-1}f(x+yI_1)+(J-I_1)(I_2-I_1)^{-1}f(x+yI_2)
	\end{equation}
	for any $J\in\mathbb{S}$, $x,y\in\mathbb{R}$ with $y\ge 0$ and $x+yI_\lambda\in U_\lambda$, $\lambda=1,2$.

By direct calculation (see the proof of \cite[Theorem 3.2]{Colombo2009001}), we find that  $g$ is slice regular on $\mathbb U^\Delta_s$ and $g=f$ on $\mathbb U^\Delta_s\cap\mathbb U^+_{s}$. Hence the function $\widetilde{f}:\mathbb U^{+\Delta}_s\rightarrow\mathbb{H}$, defined by
	\begin{equation}\label{eq-ff}
	\widetilde{f}:=\left\{
	\begin{split}
	&g,\quad &&\mbox{on}\ \mathbb U^\Delta_s,
	\\&f, &&\mbox{on}\ \mathbb U^+_{s},
	\end{split}
	\right.
	\end{equation}
	is a slice regular extension of $f|_{\mathbb U^+_{s}}$.
	
	If $h: W\longrightarrow \mathbb H$ is a slice regular extension  of $f|_{W\cap\mathbb U^+_{s}}$,
 then we have
  $$h=f=\widetilde{f}  \quad \mbox{on}\quad   W\cap \mathbb U^+_{s}$$
  so that the Identity Principle \ref{th-ip} implies
  $$h=\widetilde{f}|_W.$$
  Consequently,  $\widetilde{f}|_W$ is the unique slice regular extension on $W$ of $f|_{W\cap\mathbb U^+_{s}}$.

By \eqref{eq-im}, \eqref{eq-ff} and direct calculations, we rewrite \eqref{eq-fji} by
	\begin{equation*}
	\widetilde{f}(x+yJ)=(1,J)\mathcal{F}_{x,y}
	\end{equation*}
	for any $x,y\in\mathbb{R}$ and $J\in\mathbb{S}$ with $y\ge 0$ and $x+yK\in W$, $K=J,I_1,I_2$, where
	\begin{equation*}
	\mathcal{F}_{x,y}=
	\left(\begin{matrix}
	1&I_1\\1&I_2
	\end{matrix}\right)^{-1}
	\left(\begin{matrix}
	f(x+yI_1)\\f(x+yI_2)
	\end{matrix}\right).
	\end{equation*}

Now we can introduce a function   $G:\mathbb{R}^2\rightarrow\mathbb{H}$ defined by
	\begin{equation*}
	G(x,y):=\left\{
	\begin{split}
	&\mathcal{F}_{x,y},&&\qquad x+yI_1,x+yI_2\in W
	\\&(f(x+yI_1),0)^T,&&\qquad x+yI_1\in W\ \mbox{and}\ x+y I_{2} \notin W,
\\&(f(x+yI_2),0)^T,&&\qquad x+yI_1\notin W\ \mbox{and}\ x+y I_{2} \in W,
	\\&(0,0)^T,&&\qquad \mbox{otherwise},
	\end{split}
	\right.
	\end{equation*}
It is easy to show that $G$	is an upper stem function of $f|_W$. This means that  $\widetilde{f}$ is slice on $W$ by definition.
\end{proof}

\begin{cor}\label{ex-sb}
	 If $f:B_I(q,r)\rightarrow\mathbb{H}$ is a holomorphic function with $I\in\mathbb{S}$, $q\in\mathbb{C}_I$ and $r\in\mathbb{R}_+$, then it can be uniquely extended  to a slice regular function over  the $\sigma$-ball  $\Sigma(q,r)$.
\end{cor}

\begin{proof}   Case 1: $B_I(q,r)\cap\mathbb{R}=\varnothing$.
	
In this case, we have $$B_I(q,r)=\Sigma(q,r)$$ so that $f=\widetilde{f}$ is the unique slice regular extension of itself.
	
\medskip

	 Case 2: $B_I(q,r)\cap\mathbb{R}\neq \varnothing$.

Now we take
$$\mathbb{I}:=(I,-I)\in\mathbb{S}^2_*, \qquad \mathbb U:=(B_I(q,r),B_I(q,r))\in \tau(\mathbb{I}).$$
It is easy to see
$$
 \mathbb U^{+\Delta}_s=\Sigma(q,r),$$ which is an st-domain.  By Proposition \ref{pr-ef}, $f$ admits a unique slice regular extension $\widetilde{f}$ on $\Sigma(q,r)$.
\end{proof}

\section{Path-slice functions and representation formula}\label{sc-ps}

In this section we extend the representation formula from axially symmetric domains to non-axially-symmetric domains. To this end, we introduce
 the new notion of  path-slice functions.  It turns out that any slice regular function  on a slice-open set  is path-slice, see Theorem \ref{th-ps}.
 We can also prove the representation formula for path-slice functions.

We denote by  $\mathscr{P}(\mathbb{C})$  the set of paths $\gamma: [0,1]\longrightarrow \mathbb C$  with initial point $\gamma(0)$ in $\mathbb{R}$ and we consider its subset
\begin{equation*}
\mathscr{P}(\mathbb{C}^+):=\{\gamma\in\mathscr{P}(\mathbb{C}):\gamma(0,1]\subset\mathbb{C}^+\}.
\end{equation*}

\begin{defn}\label{df-ps}
 A function $f: \Omega\longrightarrow\mathbb H$ with   $\Omega\subset\mathbb{H}$ is called path-slice function if for any $\gamma\in\mathscr{P}(\mathbb{C})$, there is a function $F_\gamma:[0,1]\rightarrow\mathbb{H}^{2\times 1}$ such that
	\begin{equation}\label{eq-sp}
	f\circ\gamma^I=(1,I)F_\gamma
	\end{equation}
	for any $I\in\mathbb{S}$ with $\gamma^I\subset\Omega$.
	
	We call $\{F_\gamma\}_{\gamma\in\mathscr{P}(\mathbb{C})}$   a (path-)stem system of the path-slice function $f$.
\end{defn}
{Obviously, if $\Omega_\mathbb{R}=\varnothing$, then for each $\gamma\in\mathcal{P}(\mathbb{C})$, there is no $I\in\mathbb{S}$ such that $\gamma^I\subset\Omega$. Thus, by definition, every function $f:\Omega\rightarrow\mathbb{H}$ is path-slice.}

Now, we provide    equivalent characterizations  for path-slice functions.

\begin{prop}\label{pr-ps}
	For any function  $f:\Omega\rightarrow\mathbb{H}$   with  $\Omega\subset\mathbb{H}$,  the following statements are equivalent:
	\begin{enumerate}[label=(\roman*)]
		
		\item $f$ is a path-slice function.
		
		\item   For any $\gamma\in\mathscr{P}(\mathbb{C})$, there is an element  $q_\gamma\in\mathbb{H}^{2\times 1}$ such that
		\begin{equation}\label{eq-fq}
		f\circ\gamma^I(1)=(1,I)q_\gamma
		\end{equation}
		for any $I\in\mathbb{S}$ with $\gamma^I\subset\Omega$.
		
		\item   For any $\gamma\in\mathscr{P}(\mathbb{C}^+)$, there is an element  $p_\gamma\in\mathbb{H}^{2\times 1}$ such that
		\begin{equation}\label{eq-ls}
		f\circ\gamma^I(1)=(1,I)p_\gamma
		\end{equation}
		for any $I\in\mathbb{S}$ with $\gamma^I\subset\Omega$.
		
		\item   For any $\gamma\in\mathscr{P}(\mathbb{C})$ and $I,J,K\in\mathbb{S}$ with $J\neq K$ and $\gamma^I,\gamma^J,\gamma^K\subset\Omega$, we have
		\begin{equation}\label{eq-mfp}
		f\circ\gamma^I=(1,I)\left(\begin{matrix}
		1&J\\1&K
		\end{matrix}\right)^{-1}
		\left(\begin{matrix}
		f\circ\gamma^J\\f\circ\gamma^K
		\end{matrix}\right).
		\end{equation}
		
		\item   For any $\gamma\in\mathscr{P}(\mathbb{C})$ and $I,J,K\in\mathbb{S}$ with $J\neq K$ and $\gamma^I,\gamma^J,\gamma^K\subset\Omega$, we have
		\begin{equation*}
		\begin{split}
		f\circ\gamma^I=(J-K)^{-1}(Jf\circ\gamma^J-Kf\circ\gamma^K)+I(J-K)^{-1}(f\circ\gamma^J-f\circ\gamma^K).
		\end{split}
		\end{equation*}
		
		\item   For any $\gamma\in\mathscr{P}(\mathbb{C})$ and $I,J,K\in\mathbb{S}$ with $J\neq K$ and $\gamma^I,\gamma^J,\gamma^K\subset\Omega$, we have
		\begin{equation*}
		f\circ\gamma^I=(I-K)(J-K)^{-1}f\circ\gamma^J+(I-J)(K-J)^{-1}f\circ\gamma^K.
		\end{equation*}
		
	\end{enumerate}
\end{prop}

 \begin{proof} From \eqref{eq-jk} and \eqref{eq-im}, one can deduce that
assertions	(iv), (v), and (vi) are equivalent.
	
	$(i)\Rightarrow\ (iv)$. Suppose that $f$ is a path-slice function and let  $\{F_\gamma\}_{\gamma\in\mathscr{P}(\mathbb{C})}$ be its stem system. By \eqref{eq-sp} it follows that
	\begin{equation}\label{eq-fg}
	\left(\begin{matrix}
	f\circ\gamma^J\\f\circ\gamma^K
	\end{matrix}\right)
	=\left(\begin{matrix}
	1&J\\1&K
	\end{matrix}\right)F_\gamma
	\end{equation}
 for any $\gamma\in\mathscr{P}(\mathbb{C})$ and $I,J,K\in\mathbb{S}$ with $J\neq K$ and $\gamma^I,\gamma^J,\gamma^K\subset\Omega$.
It follows from  \eqref{eq-sp} and \eqref{eq-fg} that \eqref{eq-mfp} holds.

	$(iv)\Rightarrow\ (iii)$ Suppose (iv) holds.  We consider the two sets
	\begin{equation}\label{eq-a}
	\mathcal{A}:=\{\gamma\in\mathscr{P}(\mathbb{C}^+):|\{I\in\mathbb{S}:\gamma^I\subset\Omega\}|=1\}
	\end{equation}
	and
	\begin{equation}\label{eq-b}
	\mathcal{B}:=\{\gamma\in\mathscr{P}(\mathbb{C}^+):|\{I\in\mathbb{S}:\gamma^I\subset\Omega\}|>1\}.
	\end{equation}
	By the axiom of choice, there is  $(J_\gamma,K_\gamma)\in\mathbb{S}^2_*$ such that $\gamma^{J_\gamma},\gamma^{K_\gamma}\subset \Omega$ for any $\gamma\in\mathcal{B}$. We denote by $I_\gamma$ the unique imaginary unit in $\mathbb{S}$ such that $\gamma^I\in\Omega$ for any  $\gamma\in\mathcal{A}$.
	
	For any $\gamma\in\mathscr{P}(\mathbb{C}^+)$, we set
	\begin{equation*}
	p_\gamma:=\left\{
	\begin{split}
	&(f\circ\gamma^{I_\gamma},\ 0)^T,&&\qquad\gamma\in\mathcal{A},
	\\&\left(\begin{matrix}
	1&J_\gamma\\1&K_\gamma
	\end{matrix}\right)^{-1}
	\left(\begin{matrix}
	f\circ\gamma^{J_\gamma}(1)\\f\circ\gamma^{K_\gamma}(1)
	\end{matrix}\right),&&\qquad\gamma\in\mathcal{B},
	\\&(0, \ 0)^T,&&\qquad\mbox{otherwise},
	\end{split}
	\right.
	\end{equation*}
	It is immediate to verify  \eqref{eq-ls} holds.
	
	$(iii)\Rightarrow\ (ii)$
	Let $\gamma\in\mathscr{P}(\mathbb{C})$ be arbitrary. We  define
  $$s:=\max\{t\in[0,1]:\gamma(t)\in\mathbb{R}\}$$
and construct the path $\delta: [0,1]\rightarrow\mathbb{C}$ defined by
	\begin{equation*}
	\delta(t):=\left\{
	\begin{split}
	&\gamma(1),&&\qquad\gamma(1)\in\mathbb{R},
	\\&\gamma((1-s)t+s),&&\qquad \gamma(1)\in\mathbb{C}^+,
	\\&\overline{\gamma((1-s)t+s)},&&\qquad\mbox{otherwise}.
	\end{split}
	\right.
	\end{equation*}
By construction
 $$\delta\in\mathscr{P}(\mathbb{C}^+),$$
 moreover, if      $\gamma^I\subset\Omega$ for some $I\in\mathbb S$, then
 $$
   \delta^{\epsilon I}\subset\Omega,$$
 where
	\begin{equation}\label{eq-e}
	\epsilon:=\left\{
	\begin{split}
	&1,&&\qquad \gamma(1)\in \mathbb{C}^+,
	\\&-1,&&\qquad\mbox{otherwise}.
	\end{split}
	\right.
	\end{equation}

We take
	\begin{equation*}
		q_\gamma:=\left\{
		\begin{split}
			&(p_{\delta,1},\epsilon p_{\delta,2})^T&&\qquad  |\{I\in\mathbb{S}:\gamma^I\subset\Omega\}|\ge 1,
			\\&0,&&\qquad\mbox{otherwise},
		\end{split}
		\right.
	\end{equation*}
where $p_{\delta}=(p_{\delta,1},p_{\delta,2})^T\in\mathbb{H}^{2\times 1}$ is an  element satisfying  \eqref{eq-ls}, i.e.,
	\begin{equation*}\label{eq-ls875}
		f\circ\delta^I(1)=(1,I)p_{\delta}
	\end{equation*}
	for any $I\in\mathbb{S}$ with $\delta^I\subset\Omega$.
	
	Obviously, $q_\gamma$ satisfies \eqref{eq-fq} so that  (ii) holds.
	
	$(ii)\Rightarrow\ (i)$ Let $\gamma\in\mathscr{P}(\mathbb{C})$ be arbitrary and  fix a  point $t\in[0,1]$.

We consider   the   path
$\delta:[0,1]\rightarrow\mathbb{C}$, defined by
	\begin{equation*}
	\delta(s):=\gamma(ts).
	\end{equation*}
Then $\delta $ is a path  from $\gamma(0)$ to $\gamma(t)$
such that   $\delta \in \mathscr{P}(\mathbb{C})$.

Let  $q_\delta$ be an element satisfying  \eqref{eq-fq}, i.e.,
\begin{equation*}\label{eq-fq876}
		f\circ\delta^I(1)=(1,I)q_{\delta}
		\end{equation*}
	for any $I\in\mathbb{S}$ with $\delta^I\subset\Omega$.

Now we can define a  function $F_\gamma:[0,1]\rightarrow\mathbb{H}^{2\times1}$ via
	\begin{equation*}
	F_\gamma(t):=\left\{
	\begin{split}
	&q_{\delta},&&\qquad\exists\ I\in\mathbb{S}\ s.t.\ \gamma^I\subset\Omega,
	\\&(0,\ 0)^T &&\qquad \mbox{otherwise}.
	\end{split}
	\right.
	\end{equation*}
We remark that, by construction, the path $\delta$ depends on the parameter $t$.

It is direct to verify that
  \begin{equation*}\label{eq-sp359}
	f\circ\gamma^I=(1,I)F_\gamma
	\end{equation*}
	for any $I\in\mathbb{S}$ with $\gamma^I\subset\Omega$.
 This means that $f$ is path-slice since $\gamma$ is arbitrary.
\end{proof}

\begin{prop}\label{pr-sps}
	Every slice function is a path-slice function.
\end{prop}

\begin{proof}
 If $f$  is a slice function, then \eqref{eq-stf} holds.
If we set $$\gamma^I(t):=x(t)+y(t)I,$$
it is clear that  \eqref{eq-fq} follows from  \eqref{eq-stf}. This implies $f$ is path-slice.
\end{proof}

\begin{thm}\label{th-ps}
	Every slice regular function  on a slice-open set is path-slice.
\end{thm}

\begin{proof}
	Let $\Omega$ be a slice-open set and $f:\Omega\rightarrow\mathbb{H}$ be a slice regular function.
We show  that $f$ is path-slice. To this end, by Proposition \ref{pr-ps} we only need to  verify \eqref{eq-ls},
namely
we need to choose    $p_\gamma\in\mathbb{H}^{2\times 1}$ such that
		\begin{equation*}\label{eq-ls793}
		f\circ\gamma^I(1)=(1,I)p_\gamma
		\end{equation*}
for any $\gamma\in\mathscr{P}(\mathbb{C}^+)$
and $I\in\mathbb{S}$ with $\gamma^I\subset\Omega$.	We have to treat three cases.

Case 1:  Let $\mathcal B$ be as in \eqref{eq-b} and $\gamma\in\mathcal B$.

In virtue of  (\ref{eq-b}), there exist
	 $J,K\in\mathbb{S}$   such that
   $$J \neq K, \qquad \gamma^J,\gamma^K\subset\Omega.$$
   Take  $U_J$ and $U_K$ such that
   $$\gamma^J\subset U_J, \qquad \gamma^K\subset U_K, $$
   and   $U_J$  is a domain  in $\Omega_J$  and  $U_K$ is  a domain in $\Omega_K$.

  Let us set $\mathbb J=(J, K)$
and $\mathbb U=(U_J, U_K)$.
 We consider the function
   $$g=\left. f\right|_{U_J\cup U_K}.$$
This function  satisfies the conditions in the Extension Theorem \ref{pr-ef}. Therefore, $\left. g\right|_{\mathbb U_s^+}$ has a slice regular  extension
$\widetilde g$  over the slice-connected component $W$ of $\mathbb U_{s, \mathbb J}^{+\Delta}\cap \Omega$ containing $\gamma(0)$.
By the Identity Principle, see Theorem \ref{th-ip}, we have
        $$f=\widetilde g \qquad \mbox{on}\quad  W.$$
Since $\widetilde g$ is slice on $W$, it follows that  $f$ is slice on $W$.

Recall that $\gamma\in\mathscr{P}(\mathbb{C}^+)$. By construction we have
$$\gamma^J, \gamma^K\subset \mathbb U_{s, \mathbb J}^{+\Delta}.$$
This implies that for any $L\in\mathbb{S}$
     $$\gamma^L \subset \mathbb U_{s, \mathbb J}^{+\Delta}.$$

Then for any $I\in\mathbb{S}$ with $\gamma^I\subset\Omega$,
     $$\gamma^I \subset \mathbb U_{s, \mathbb J}^{+\Delta}\cap \Omega.$$
  Since $\gamma^I(0)\in W$ and  $W$      is a slice-connected component   of $\mathbb U_{s, \mathbb J}^{+\Delta}\cap \Omega$,
     we thus conclude $$\gamma^I\subset W.$$

Due to the fact that   $f$ is slice on $W$,
       Proposition \ref{pr-sfe} (ii) implies that there exists    a function $F_\gamma:\mathbb{R}^2\rightarrow\mathbb{H}^{2\times 1}$ such that
	\begin{equation}\label{eq-fxy}
	f(x_\gamma+y_\gamma I)=(1,I)F_\gamma(x_\gamma,y_\gamma)
	\end{equation}
for any $I\in\mathbb{S}\mbox{ with }\gamma^I\subset\Omega$, where we have written  $$\gamma(1)=x_\gamma+y_\gamma i$$ for some $x_\gamma,y_\gamma\in\mathbb{R}.$

Finally, we set
   	\begin{equation}\label{eq-pg}
	p_\gamma:=
	F_\gamma(x_\gamma,y_\gamma), \qquad\gamma\in\mathcal{B}.
	\end{equation}

Case 2. Let $\mathcal A$ be as in \eqref{eq-a} and $\gamma\in\mathcal{A}. $

In this case, we take

	\begin{equation}\label{eq-pg423}
	p_\gamma:=
 	 (f\circ\gamma^{I_\gamma}(1),0)^T, \qquad\gamma\in\mathcal{A},
	\end{equation}
where $I_\gamma$ is the unique imaginary unit $I\in\mathbb{S}$ such that $\gamma^{I_\gamma}\subset\Omega$.

Case 3. Let $\gamma\notin\mathcal{A}\cup\mathcal B.$

In this case, we take $p_\gamma:=(0, 0)^T$.

\medskip

With the choice of $p_\gamma$ above,
		it is  clear that $p_\gamma$ satisfies \eqref{eq-ls} as desired.
\end{proof}

\begin{proof}[Proof of Theorem \ref{th-rf}]
	It is a direct consequence of  Proposition \ref{pr-ps} and Theorem \ref{th-ps}.
\end{proof}

\begin{rmk}
A slice regular function on a slice-open set is not necessarily a slice function. To provide an example, let us fix $J\in\mathbb{S}$ and consider $f:\mathbb{H}\backslash\mathbb{R}\rightarrow\mathbb{H}$, defined by
\begin{equation*}
	f(q)=\begin{cases}
		0,& q\in\mathbb{C}_J\backslash\mathbb{R}
		\\1,& \mbox{otherwise}.
	\end{cases}
\end{equation*}
Then $f$ is a slice regular function defined on the slice-open set $\mathbb{H}\backslash\mathbb{R}$ but it is not a slice function. Indeed, $f$ is a constant in each slice $\mathbb{C}_I$, $I\in\mathbb{C}_I$, so that $f$ is a slice regular function. Suppose that $f$ is a slice function. If $x+yI\in\mathbb{H}\backslash\mathbb{R}$ with $I\neq\pm J$, then $f(x+yJ)=f(x-yJ)=0$ and we would have
\begin{equation*}
	f(x+yI)=(1,I)\begin{pmatrix}
		1&J\\1&-J
	\end{pmatrix}^{-1}\begin{pmatrix}
	f(x+yJ)\\f(x-yJ)
\end{pmatrix}=0.
\end{equation*}
However $x+yI\notin\mathbb{C}_J\backslash\mathbb{R}$, so that $f(x+yI)=1$, which is a contradiction.
\end{rmk}

\begin{prop}\label{rm-scs} The set of  slice functions and the set of path-slice functions on axially symmetric slice-path-connected set which intersects with $\mathbb{R}$ contain the same elements.
\end{prop}

\begin{proof}

Let $\Omega$ be an axially symmetric path-slice-connected set with $\Omega_{\mathbb{R}}\neq\varnothing$.
According to Proposition \ref{pr-sps}, we just need to prove that any path-slice function on $\Omega$ is slice. Let $f:\Omega\rightarrow\mathbb H$ be a path-slice function and let us prove that $f$ is slice.
Since $\Omega$ is path-slice-connected, for any $z\in\Omega_\mathbb{R}$ and $q \in\Omega$, there is a slice-path $\alpha$ from $z$ to $q$.

We write $q=x+yI$ for some $x,y\in\mathbb{R}$ and $I\in\mathbb{S}$. Since
	\begin{equation*}
	\mathbb{H}\backslash\mathbb{C}_I:=\bigcup_{J\in\mathbb{S}\backslash\{\pm I\}}\mathbb{C}_J^+
	\end{equation*}
	is slice-open, the preimage $\alpha^{-1}(\mathbb{C}_I)$ is closed in $[0,1]$. We write by $[t,1]$ the connected component of $\alpha^{-1}(\mathbb{C}_I)$ containing $1$ for some $t\in[0,1]$.

We consider the path $\gamma:[0,1]\rightarrow\mathbb{C}$,  defined by
	\begin{equation*}
	\gamma(s):
	=\mathcal{P}_I^{-1}\circ\alpha(t+(1-t)s).	
	\end{equation*}
	It is clear that $\gamma$ is in $\mathscr{P}(\mathbb{C})$ and $\gamma^I$ is a path from $\alpha(t)$ to $q$.
	
	Since $\Omega$ is axially symmetric, we have $$\gamma^J\subset\Omega$$ for any $J\in\mathbb{S}$. Since
	\begin{equation*}
	\gamma^J(1)=x+yJ,\qquad\forall\ J\in\mathbb{S},
	\end{equation*}
	it follows from  Proposition \ref{pr-ps} (iv) that
	\begin{equation*}
	f(x+yL)=(1,L)\left(\begin{matrix}
	1&J\\1&K
	\end{matrix}\right)^{-1}
	\left(\begin{matrix}
	f(x+yJ)\\f(x+yK)
	\end{matrix}\right)
	\end{equation*}
for any $L,J,K\in\mathbb{S}$ with $J\neq K$.
	This implies that  $f$ is slice by Proposition \ref{pr-sfe} (iii).
\end{proof}

\begin{rmk}
{
We note that the set of axially symmetric
slice-path-connected open sets with intersection with $\mathbb{R}$ coincide with the set of axially
symmetric s-domains. In fact, it is immediate that any axially symmetric s-domain $\Omega$ is an axially symmetric slice-path-connected open set. To prove the converse statement, we consider an axially symmetric
slice-path-connected open set  $\Omega$ with $\Omega_{\mathbb R}\not=\varnothing$, and we show that it is an axially symmetric s-domain. First of all, we note that for any $I\in\mathbb{S}$ and any $p,q\in\Omega_\mathbb{R}$, there is a path in $\Omega_I$ from $p$ to $q$.
We define the map $\mathcal{P}:\mathbb{H}\rightarrow\mathbb{C}$ by}
{
\begin{equation*}
	\mathcal{P}(x+yJ):=\begin{cases}
		x,&\qquad y=0,\\
		x+|y|i,&\qquad y\neq 0.
	\end{cases}
\end{equation*}
Since $\Omega$ is slice-path-connected, there is a slice-path $\gamma$ from $p$ to $q$. Then $[\mathcal{P}(\gamma)]^I$ is a path in $\Omega_I$ from $p$ to $q$.
 Let now $p\in\Omega_I$ and $q\in\Omega_\mathbb{R}$, one may show that there is a path $\gamma$ in $\mathbb{C}_I$ from $p$ to $q$ reasoning as in the proof of Proposition \ref{pr-rcsd} (ii) and similarly, for any $p,q\in\Omega_I$, there is a path in $\mathbb{C}_I$ from $p$ to $q$.  Hence $\Omega_I$ is a domain and so $\Omega$ is an axially symmetric s-domain.
}

\end{rmk}

\begin{rmk}
	Suppose that  $\Omega$ in Theorem \ref{th-rf} is  an axially symmetric $s$-domain in $\mathbb{H}$. For any $q=x+yI\in\Omega$,  there exists a point $p\in\Omega_\mathbb{R}$ and a path $\gamma$ in $\mathbb{C}$ such that $\gamma^I$ is a path from $p$ to $q$. Since  $\Omega$ is axially symmetric, we know that,  for all $K\in\mathbb{S}$,  $\gamma^K\subset\Omega$ and $\gamma^K(1)=x+yK$. By Theorem \ref{th-rf}, we have
	\begin{equation}\label{eq-frf}
	f(x+yI)=(I-K)(J-K)^{-1}f(x+yJ)+(I-J)(K-J)^{-1}f(x+yK)
	\end{equation} for any $J, K\in \mathbb S$ with $J\not=K$.
This means that Theorem \ref{th-rf} recovers  the classical   representation formula \cite[Theorem 3.2]{Colombo2009001}.
\end{rmk}

\section{Counterexample on non-axially symmetric domains}\label{sc-ex}

In this section, we give an example to illustrate that the classical representation formula may not hold  for non-axially symmetric domains.

Let  $s\in[0,1]$  be fixed. Define a ray  $\gamma_s: [0, 1)\longrightarrow  \mathbb{C}$  by
\begin{equation*}
\gamma_s (t):=\frac{i}{2}+\frac{t}{1-t}e^{i(\frac{\pi }{4}+\frac{s\pi }{2})}.
\end{equation*}
Geometrically, the ray starts from $\frac{i}{2}$ to $\infty$  and the angle between the ray and
the positive real axis
 is $\frac{\pi}{4}+s\frac{\pi}{2}$.

For any continuous function  $$\varphi:\mathbb{S}\rightarrow[0,1],$$
we define a continuous function $F:\mathbb{S}\times[0,1)\rightarrow\mathbb{H}$ by
\begin{equation*}
F(I, s)=\mathcal{P}_I\circ\gamma_{\varphi(I)}(s).
\end{equation*}
The complement of  the image of $F$ is denoted by
\begin{equation*}
\Omega_{\varphi}:=\mathbb{H}\backslash F(\mathbb{S}\times[0,1)).
\end{equation*}

\begin{prop} The set $\Omega_{\varphi}$
	 	is a s-domain and  $$\Omega_\varphi\cap\mathbb{S}=\mathbb{S}\backslash\varphi^{-1}(\frac{1}{2}).$$
\end{prop}

\begin{proof}
	(i) For any $I\in\mathbb{S}$, we denote
	  $$\gamma_\varphi[I]:=\mathcal{P}_I\circ\gamma_{\varphi(I)}([0,1)).$$
Then  $\gamma_\varphi[I]$  is an image of a ray in $\mathbb{C}_I$ from $\frac{I}{2}$ to $\infty$. And the angle between the ray and
the positive real axis
 is
	$$\frac{\pi}{4}+\varphi(I)\frac{\pi}{2}.$$

By definition,
\begin{equation*}
F(\mathbb{S}\times[0,1))=\bigcup_{I\in\mathbb{S}}\gamma_\varphi[I]
\end{equation*}
and
	\begin{equation}\label{eq:def:omega-phi-120}
	\Omega_{\varphi}=\mathbb{H}\backslash\bigcup_{I\in\mathbb{S}}\gamma_\varphi[I].
	\end{equation}

 Since
 \begin{equation*}
\mathcal{P}_I\circ \gamma_{\varphi(I)} (t):=\frac{I}{2}+\frac{t}{1-t}e^{\frac{\varphi(I)\pi I}{2}+\frac{\pi I}{4}}, \qquad \forall \ t\in[0,1),
\end{equation*}
 we have
  $$\left[\bigcup_{K\in\mathbb{S}}B(\frac{K}{2},\frac{t}{1-t})\right]\bigcap \mathcal{P}_I\circ \gamma_{\varphi(I)} [0, 1)=\mathcal{P}_I\circ \gamma_{\varphi(I)} [0, t)$$
  for any $t\in (0, 1)$ and $I\in \mathbb S$.
  Taking the union for all $I\in \mathbb S$, we get
  \begin{equation}\label{eq-bff}
\left[\bigcup_{K\in\mathbb{S}}B(\frac{K}{2},\frac{t}{1-t})\right]\bigcap F(\mathbb{S}\times[0,1))=F(\mathbb{S}\times[0,t)).
\end{equation}

Denote
\begin{equation}\label{eq-atb}
A_t:=\left[\bigcup_{K\in\mathbb{S}}B(\frac{K}{2},\frac{t}{1-t})\right]\bigcap(\mathbb{H}\backslash F(\mathbb{S}\times[0,t])).
\end{equation}
Since $F$ is continuous, we have  $F(\mathbb{S}\times [0,t])$ is compact so that $A_t$ is open.

By \eqref{eq-bff} and \eqref{eq-atb}, we have
\begin{equation*}
A_t=\left[\bigcup_{K\in\mathbb{S}}B(\frac{K}{2},\frac{t}{1-t})\right]\bigcap [\mathbb{H}\backslash F(\mathbb{S}\times[0,1))]=\left[\bigcup_{K\in\mathbb{S}}B(\frac{K}{2},\frac{t}{1-t})\right]\bigcap\Omega_\varphi
\end{equation*}
and
\begin{equation*}
	\begin{split}
		\bigcup_{t\in(0,1)}A_t=&\bigcup_{t\in(0,1)}\left(\left[\bigcup_{K\in\mathbb{S}}B(\frac{K}{2},\frac{t}{1-t})\right]\bigcap\Omega_\varphi\right)
		\\=&\left(\bigcup_{t\in(0,1)}\left[\bigcup_{K\in\mathbb{S}}B(\frac{K}{2},\frac{t}{1-t})\right]\right)\bigcap\Omega_\varphi
		=\mathbb{H}\bigcap\Omega_\varphi=\Omega_\varphi.
	\end{split}
\end{equation*}
This means that $\Omega_{\varphi}$ is open.

Note that $(\Omega_\varphi)\cap \mathbb C_I$ is $\mathbb C_I$ deleting two rays. One is  emitting from $I/2$ lying in the upper space $\mathbb C^+_I$, while the other  one
 emitting from $-I/2$ lying in the lower  space $\mathbb{C}^-_I:=\mathbb C^+_{-I}$.
Therefore,  $(\Omega_\varphi)_I$ is a domain in $\mathbb{C}_I$ and path-connected. And since $\Omega_\varphi\cap\mathbb R=\mathbb R$, $\Omega_\varphi$ is a s-domain.
	
	(ii) Note that for any $I\in\mathbb{S}$, $I\in\gamma_\varphi[I]$ if and only if $\varphi(I)=\frac{1}{2}$. It  follows that $$\Omega_\varphi\cap\mathbb{S}=\mathbb{S}\backslash\varphi^{-1}(\frac{1}{2}).$$
\end{proof}

Let us now fix $J\in\mathbb{S}$.
The classical theory of holomorphic functions shows that the function
 \begin{equation}\label{eq-fp647}
\Psi (z)=\sqrt{2z-J}, \qquad  \forall\  z\in \frac{J}{2}+\mathbb{R}_+,
\end{equation}
 admits  a unique holomorphic extension $\Psi_s$ on $$\mathbb{C}_J\backslash(\gamma_s[J]\cup\gamma_s[-J]), $$
where
  $$\gamma_s[J]:=\mathcal{P}_J\circ\gamma_s([0,1))$$
 for any $s\in[0,1]$.

\begin{rmk}\label{rm-fs}
	The function $\Psi_s$ has the following properties.
\\
(i) For any $s,t\in[0,1]$,
$$\left.\Psi_s\right|_{\mathbb R}=\left.\Psi_t\right|_{\mathbb R}.$$
\\		
(ii) For any $s\in[0,1]$, we have
$$\Psi_s(-J)=\sqrt{3}e^{-\frac{J\pi}{4}}$$
and
		\begin{equation*}
		\Psi_s(J)=\left\{\begin{split}
			&-e^{\frac{J\pi}{4}},&\qquad s\in[0,\frac{1}{2}),\\&e^{\frac{J\pi}{4}},&\qquad s\in(\frac{1}{2},1].
			\end{split}\right.
		\end{equation*}
\\
(iii) For any $s\in[0,1]$, denote $\alpha:=\frac{\pi}{2}+\frac{s\pi}{2}$.  Then for any $\lambda\in\mathbb{R}_+$
\begin{equation*}
\lim_{\theta\rightarrow\alpha^-}\Psi_s(\frac{J}{2}+\lambda e^{\theta J})=\sqrt{\lambda} e^{\frac{\alpha J}{2}}
\end{equation*}
and
\begin{equation*}
\lim_{\theta\rightarrow\alpha^+}\Psi_s(\frac{J}{2}+\lambda e^{\theta J})=-\sqrt{\lambda} e^{\frac{\alpha J}{2}}.
\end{equation*}
This implies  that $\Psi_s$ cannot be  extended continuously to any point in $\gamma_s[J]\backslash\{\frac{J}{2}\}$.		
\end{rmk}

\begin{prop}
	Let $\varphi:\mathbb{S}\rightarrow[0,1]$ be a continuous function and $\Psi$ is as in \eqref{eq-fp647}. Then the function $\Psi_\varphi:\Omega_\varphi\rightarrow\mathbb{H}$ defined by
	\begin{equation}\label{eq-fp}
	\Psi_\varphi(x+yI):=
		\frac{1-IJ}{2}\Psi_{\varphi(I)}(x+yJ)+\frac{1+IJ}{2}\Psi_{\varphi(I)}(x-yJ),
	\end{equation}
	for $y\ge 0$, is the unique slice regular extension of $\Psi$ to $\Omega_\varphi$. In particular, $$(\Psi_\varphi)_J=\Psi_{\varphi(J)}.$$
\end{prop}

\begin{proof}
	By direct calculation, $(\Psi_\varphi)_I$ is a holomorphic  extension of $\Psi_{\varphi(I)}|_\mathbb{R}$. And $\Psi_\varphi$ is well-defined by Remark \ref{rm-fs} (i). It is clear that $\Psi_\varphi$ is the unique slice regular extension of $\Psi$.
\end{proof}

\begin{prop}
	Formula \eqref{eq-frf} does not hold, in general, for slice regular functions defined on a non-axially symmetric domains.
\end{prop}

\begin{proof}
To show that the statement does not hold in general, we provide a counterexample. Let us recall that $J\in\mathbb{S}$ is fixed in \eqref{eq-fp647} and in this proof.
Let $\varphi(K)=\frac{1}{2}|K-J|$, then (using the above notations)
	\begin{equation*}
		\Psi_\varphi(I)\neq\frac{1-IJ}{2}\Psi_\varphi(J)+\frac{1+IJ}{2}\Psi_\varphi(-J),
	\end{equation*}
	for each $I\in\mathbb{S}$ with $\frac{1}{2}<\varphi(I)<1$.
In fact, since $\varphi(-J)=1$, we have $I\neq -J$ and $(1-IJ)\neq 0$. By Remark \ref{rm-fs} (ii),
	\begin{equation*}
		\Psi_\varphi(J)=-\Psi_{\varphi(I)}(J)\neq 0\qquad\mbox{and}\qquad\Psi_\varphi(-J)=\Psi_{\varphi(I)}(-J).
	\end{equation*}
	From \eqref{eq-fp} we obtain that
	\begin{equation*}
		\Psi_\varphi(I)-\left[\frac{1-IJ}{2}\Psi_\varphi(J)+\frac{1+IJ}{2}\Psi_\varphi(-J)\right]=(1-IJ)\Psi_{\varphi(I)}(J)\neq 0.
	\end{equation*}
\end{proof}

\begin{defn}
	Let $\Omega\subset\mathbb{H}$. A function $f:\Omega\rightarrow\mathbb{H}$ is called slice-Euclidean continuous if for any $U\in\tau(\mathbb H)$, the preimage $f^{-1}(U)$ is slice-open. In other words, $f:(\Omega, \tau_s(\mathbb H))\rightarrow(\mathbb{H, \tau(\mathbb H)})$ is continuous.
\end{defn}

\begin{prop}\label{prop-sec}
	Let $\Omega\subset\mathbb{H}$. A function $f:\Omega\rightarrow\mathbb{H}$ is slice-Euclidean continuous if and only if for any $I\in\mathbb{S}$, $f_I$ is continuous.
\end{prop}

\begin{proof}
	Let  $I\in\mathbb{S}$ and $U\in\tau(\mathbb H)$. If  $f_I$ is continuous, then $(f_I)^{-1}(U)$ is open in $\mathbb{C}_I$.
Hence
	\begin{equation*}
	f^{-1}(U)=\bigcup_{I\in\mathbb{S}}(f_I)^{-1}(U)
	\end{equation*}
	is slice-open.

Conversely,   if $f:\Omega\rightarrow\mathbb{H}$ is slice-Euclidean continuous, then for any  $U\in\tau(\mathbb H)$ $f^{-1}(U)\in \tau_s(\mathbb H)$. This means
$(f_I)^{-1}(U)$ is open in $\mathbb C_I$ for any $I\in\mathbb S$.
Therefore, $f_I$ is continuous.
\end{proof}

\begin{prop}\label{pr-sec}
	Every slice regular function is slice-Euclidean continuous.
\end{prop}

\begin{proof}
	This follows directly from  Proposition \ref{prop-sec}.
\end{proof}

\begin{prop}\label{pr-sre} Let $\Psi_\varphi$ be as in \eqref{eq-fp}. For any continuous function $\varphi:\mathbb{S}\rightarrow[0,1]$,
	there is a unique slice regular extension $\widetilde{\Psi_\varphi}$ of $\Psi_\varphi$ on
	\begin{equation*}
	\widetilde{\Omega_\varphi}:=\Omega_\varphi\bigcup\gamma_\varphi[-J].
	\end{equation*}
		Moreover, $\widetilde{\Psi_\varphi}$ cannot be extended slice-Euclidean continuously to any point in $(\mathbb{H}\backslash\widetilde{\Omega_\varphi})\cup{(\mathbb{H}\backslash\frac{\mathbb{S}}{2})}$.
\end{prop}

\begin{proof}
	Note that $\Psi_\varphi(q)=\sqrt{2q-J}$ for any $q\in\frac{J}{2}+\mathbb{R}^+$. Form complex analysis, we know  that $\Psi_\varphi$ can be extended slice regularly  to $\gamma_\varphi[-J]$. This extension,  denoted by  $\widetilde{\Psi_\varphi}$,    is unique by Identity Principle \ref{th-ip}.  For any $\lambda\in\mathbb{R}^+$, we have
	\begin{equation}\label{eq-jbt}
	\lim_{\theta\rightarrow\beta}\Psi_\varphi(\frac{-J}{2}+\lambda e^{-J\theta})=\widetilde {\Psi_\varphi}(\frac{-J}{2}+\lambda e^{-J\beta})
	\end{equation}
	where $$\beta:=\frac{\pi}{2}+\frac{\varphi(-J)\pi}{2}.$$
	
	For any $I\in\mathbb{S}\backslash\{-J\}$, denote $$\alpha:=\frac{\pi}{2}+\frac{\varphi(I)\pi}{2}.$$  It follows from \eqref{eq-fp} and \eqref{eq-jbt} that for any $\lambda\in\mathbb{R}_+$
	\begin{equation*}
	\lim_{\theta\rightarrow\alpha-} \widetilde{\Psi_\varphi}(\frac{I}{2}+\lambda e^{I\theta})=\frac{1-IJ}{2}\lim_{\theta\rightarrow\alpha-}\Psi_{\varphi(I)}(\frac{J}{2}+\lambda e^{J\theta})+\frac{1+IJ}{2}\widetilde{\Psi_\varphi}(\frac{-J}{2}+\lambda e^{-J\alpha})
	\end{equation*}
	and
	\begin{equation*}
	\lim_{\theta\rightarrow\alpha+} \widetilde{\Psi_\varphi}(\frac{I}{2}+\lambda e^{I\theta})=\frac{1-IJ}{2}\lim_{\theta\rightarrow\alpha+}\Psi_{\varphi(I)}(\frac{J}{2}+\lambda e^{J\theta})+\frac{1+IJ}{2}\widetilde{\Psi_\varphi}(\frac{-J}{2}+\lambda e^{-J\alpha}).
	\end{equation*}
	By Remark \ref{rm-fs} (iii), we find
	\begin{equation*}
	\lim_{\theta\rightarrow\alpha-} \widetilde{\Psi_\varphi}(\frac{I}{2}+\lambda e^{I\theta})\neq
	\lim_{\theta\rightarrow\alpha+} \widetilde{\Psi_\varphi}(\frac{I}{2}+\lambda e^{I\theta}).
	\end{equation*}

	And according to Proposition \ref{prop-sec}, $\widetilde{\Psi_\varphi}$ can not be extended slice-Euclidean continuously to any point in $\gamma_\varphi[I]\backslash\{\frac{I}{2}\}$. Since
	\begin{equation*}
	\bigcup_{I\in\mathbb{S}\backslash\{-J\}}\left(\gamma_\varphi[I]\backslash\{\frac{I}{2}\}\right)=(\mathbb{H}\backslash\widetilde{\Omega_\varphi})\cup{(\mathbb{H}\backslash\frac{\mathbb{S}}{2})},
	\end{equation*}
	it follows that $\widetilde{\Psi_\varphi}$ cannot be extended slice-Euclidean continuously to any point in $(\mathbb{H}\backslash\widetilde{\Omega_\varphi})\cup{(\mathbb{H}\backslash\frac{\mathbb{S}}{2})}$.
\end{proof}

\begin{prop}\label{pr-fsd}
	$\widetilde{\Psi_\varphi}$ cannot be slice regularly extended to any st-domain containing strictly $\widetilde{\Omega_\varphi}$.
\end{prop}

\begin{proof}
	This is a direct consequence of Propositions \ref{pr-sec} and \ref{pr-sre}.
\end{proof}

\begin{rmk}\label{exa-sre}
	Notice that $\Omega_\varphi$ is not axially symmetric when $\varphi$ is not constant. Moreover the only axially symetric st-open set including $\Omega_\varphi$ is $\mathbb{H}$, since
	\begin{equation*}
		\bigcup_{x+yI\in\Omega_{\varphi}} x+y\mathbb{S}=\mathbb{H}.
	\end{equation*}
	By Remark \ref{rmk-sd} and Proposition \ref{pr-fsd}, $\Psi_\varphi$ cannot be slice regularly extended to any axially symmetric s-domain in $\mathbb{H}$,  when $\varphi(K)=\frac{1}{2}|K-J|$, for $K\in\mathbb{S}$.
\end{rmk}

	We now provide an example of a slice regular function defined on a slice-open set $\Omega\in\tau_s\backslash\tau_\sigma$. The fact that this is indeed an example as required can be seen following the reasonings in this section.
	\begin{exa}
		For each $s\in(0,2]$, define
		\begin{equation*}
			W_s:=\{x+yi:x=0\mbox{ and }\frac{s}{8}\le y\le\frac{1}{2}\}\cup\{x+yi:x\le0\mbox{ and }y=\frac{s}{8}\}
		\end{equation*}
		and $W_0:=\varnothing$. Fix $I\in\mathbb{S}$. Define
		\begin{equation*}
			\Omega:=\bigcup_{J\in\mathbb{S}}\left[\mathbb{C}_J^+\backslash\mathcal{P}_J(W_{|J+I|})\right]\cup\mathbb{R}.
		\end{equation*}
		Then $\Omega\in\tau_s\backslash\tau_\sigma$. The function $\Psi$ defined in \eqref{eq-fp647} can be extended to a slice regular function $\Psi'$ on $\Omega$. And $\Psi'$ can not be extended slice regularly to any slice-open set containing strictly $\Omega$.
	\end{exa}

\section{Domains of slice regularity}\label{sc-dsr}

In this section, we consider domains of slice regularity for slice regular functions, analogous to domains of holomorphy of holomorphic functions.
It turns out that
  the $\sigma$-balls and axially symmetric slice-open sets are domains of slice regularity.

In contrast to complex analysis of one variable,
an st-domain may fail to be a domain of slice regularity.

 We also give a property of domains of slice regularity, see Proposition \ref{pr-dsr}.

\begin{defn}\label{df-dsr}
	A slice-open set $\Omega\subset\mathbb{H}$ is called a domain of slice regularity if there are no slice-open sets $\Omega_1$ and $\Omega_2$ in $\mathbb{H}$ with the following properties.
	\begin{enumerate}[label=(\roman*)]
		\item $\varnothing\neq\Omega_1\subset\Omega_2\cap\Omega$.
		\item $\Omega_2$ is slice-connected and not contained in $\Omega$.
		\item For any slice regular function $f$ on $\Omega$, there is a slice regular function $\widetilde{f}$ on $\Omega_2$ such that $f=\widetilde{f}$ in $\Omega_1$.
	\end{enumerate}
Moreover, if there are slice-open sets $\Omega,\Omega_1,\Omega_2$ satisfying (i)-(iii), then we call $(\Omega,\Omega_1,\Omega_2)$ a slice-triple.
\end{defn}

{
In a similar way, we give the following:
\begin{defn}
	Let $\Omega$ be a slice-open set, $I\in\mathbb{S}$ and $U_1$, $U_2$ be open sets in $\mathbb{C}_I$. $(\Omega,U_1,U_2)$ is called an $I$-triple if
	\begin{enumerate}[label=(\roman*)]
		\item $\varnothing\neq U_1\subset U_2\cap\Omega_I$.
		\item $U_2$ is connected in $\mathbb{C}_I$ and not contained in $\Omega_I$.
		\item For any slice regular function $f$ on $\Omega$, there is a holomorphic function $\widetilde{f}:U_2\rightarrow\mathbb{H}$ such that $f=\widetilde{f}$ in $U_1$.
	\end{enumerate}
\end{defn}

\begin{lem}\label{pr-ubs}
	Let $U$ be a slice-open set and $\Omega$ be an st-domain with $U\subsetneq \Omega$. Then $\Omega\cap\partial_I U_I\neq\varnothing$ for some $I\in\mathbb{S}$, where $\partial_I U_I$ is the boundary of $U_I$ in $\mathbb{C}_I$.
\end{lem}

\begin{proof}
	Suppose that $\Omega\cap\partial_I U_I=\varnothing$ for each $I\in\mathbb{S}$. Since $\Omega_I$ and $\mathbb{C}_I\backslash(\partial_I U_I\cup U_I)$ are open in $\mathbb{C}_I$, so is
	\begin{equation*}
		\begin{split}
			[\Omega\cap(\mathbb{H}\backslash U)]_I
			&=\Omega_I\cap(\mathbb{C}_I\backslash U_I)
			=\Omega_I\backslash(\Omega_I\cap U_I)
			\\&=\Omega_I\backslash[\Omega_I\cap (\partial_IU_I\cup U_I)]
			=\Omega_I\cap[\mathbb{C}_I\backslash(\partial_I U_I\cup U_I)].
		\end{split}
	\end{equation*}
	By definition, $\Omega\cap(\mathbb{H}\backslash U)$ is slice-open. Hence $\Omega$ is the disjoint union of the nonempty slice-open sets $\Omega\cap(\mathbb{H}\backslash U)$ and $\Omega\cap U$. It implies that $\Omega$ is not slice-connected, a contradiction.
\end{proof}

\begin{prop}\label{pr-so}
	A slice-open set $\Omega\subset\mathbb{H}$ is a domain of slice regularity if and only if for any $I\in\mathbb{S}$ there are no open sets $U_1$ and $U_2$ in $\mathbb{C}_I$ such that $(\Omega,U_1,U_2)$ is an $I$-triple.
\end{prop}

\begin{proof} Let $\Omega$ be a domain of slice regularity and let us suppose, by absurd, that there exists an $I$-triple  $(\Omega,U_1,U_2)$ for $I\in\mathbb{S}$.

Let $V\subset U_1$ be a nonempty domain in $\mathbb{C}_I$ such that $V\cap\mathbb{R}=\varnothing$, and choose $J\in\{\pm I\}$ such that $V\subset\mathbb{C}_{J}^+$. It is clear that $V$ is an st-domain and, by definition, $(\Omega,V,U_2)$ is a $J$-triple.
	Let $f:\Omega\rightarrow\mathbb{H}$ be a slice regular function. By Theorem \ref{pr-ef}, where we take
$\mathbb{I}:=(J,-J)$, and $\mathbb{U}:=(U_2,U_2)$, we deduce that $f|_V$ can be extended to a slice regular function $\widetilde f$ on a slice-open set $\mathbb{U}^{+\Delta}_{s,\mathbb{I}}\supset U_2$. Let $\widetilde{V}$ be the slice-connected component of $\mathbb{U}^{+\Delta}_{s,\mathbb{I}}$ containing $V$. Since $U_2\supset V$ is connected in $\mathbb{C}_{J}$, we have $\widetilde{V}\supset U_2$. Since $U_2\nsubseteq\Omega_{J}$, we have $\widetilde{V}\nsubseteq\Omega_J$, and then $\widetilde{V}\nsubseteq\Omega$. Thus we have
	that $\varnothing\neq V\subset\widetilde{V}\cap\Omega$, $\widetilde{V}$ is slice-connected and not contained in $\Omega$. Moreover, for any slice regular function $f$ on $\Omega$, there is a slice regular function $\widetilde f|_{\widetilde V}$ on $\widetilde{V}$ such that $f=\widetilde f$ on $V$.
	
We conclude that $(\Omega,V,\widetilde{V})$ is a slice-triple, and $\Omega$ is not a domain of slice regularity, which is a contradiction.

Now we prove the converse, i.e. a slice-open set $\Omega$ is a domain of slice regularity if for each $I\in\mathbb{S}$ there are no open sets $U_1$ and $U_2$ in $\mathbb{C}_I$ such that $(\Omega,U_1,U_2)$ is an $I$-triple.
	So we suppose that $\Omega$ is not a domain of slice regularity. Then there are slice-open sets $\Omega_1,\Omega_2$ such that $(\Omega,\Omega_1,\Omega_2)$ is a slice-triple. Let $U$ be a slice-connected component of $\Omega\cap\Omega_2$ with $U\cap\Omega_1\neq\varnothing$. By the Identity Principle \ref{th-ip} we have that $(\Omega,U,\Omega_2)$ is also a slice-triple.
	
We claim that for any $I\in\mathbb{S}$, $U_I$ is a union of some connected components of $\Omega_I\cap(\Omega_2)_I$ in $\mathbb{C}_I$. (This follows from the general fact that if $\Sigma$ be a slice-open set and $U$ be a slice-connected component of $\Sigma$. Then for each $I\in\mathbb{S}$, $U_I$ is a union of some connected components of $\Sigma_I$).  Then for any $I\in\mathbb{S}$,
	\begin{equation*}
		\partial_I U_I\subset\partial_I((\Omega_2)_I\cap\Omega_I)\subset\partial_I((\Omega_2)_I)\cup\partial_I(\Omega_I).
	\end{equation*}
	Since $(\Omega_2)_I\cap\partial_I((\Omega_2)_I)=\varnothing$,
	\begin{equation}\label{eq-o2}
		(\Omega_2)_I\cap\partial_I U_I\subset\partial_I\Omega_I.
	\end{equation}

	By Lemma \ref{pr-ubs}, $(\Omega_2)_J\cap\partial_J U_J\neq\varnothing$ for some $J\in\mathbb{S}$. Let $p\in(\Omega_2)_J\cap\partial_J U_J$. By \eqref{eq-o2}, $p\in\partial_J\Omega_J$, and then $p\notin \Omega_J$. Let $\Omega_3$ be the connected component of $(\Omega_2)_J$ containing $p$ in $\mathbb{C}_J$, and $U':=U_J\cap\Omega_3$. Since $p\in\partial_J U_J$ and $p$ is an interior point of $\Omega_3$ in $\mathbb{C}_J$, we have $p\in\partial_J U'$ and then $U'\neq\varnothing$. Hence
	
	(i) $\varnothing\neq U'\subset \Omega_3\cap\Omega_J$.
	
	(ii) $\Omega_3$ is connected in $\mathbb{C}_J$ and not contained in $\Omega_J$ (by $p\in\Omega_3$ and $p\notin\Omega_J$).
	
	(iii) Since $(\Omega,\Omega_1,\Omega_2)$ is a slice-triple, for any slice regular function $f$ on $\Omega$, there is a slice regular function $f':\Omega_2\rightarrow\mathbb{H}$ such that $f=f'$ in $\Omega_1$. Since $\Omega_3\subset\Omega_2$ and $U'\subset U\subset\Omega_1$, we have $f'|_{\Omega_3}$ is a  holomorphic function such that $f=f'|_{\Omega_3}$ on $U'$.
	It implies that $(\Omega,U',\Omega_3)$ is a $J$-triple, a contradiction and the assertion follows.

\end{proof}}

\begin{prop}\label{pr-aa}
	Any axially symmetric {slice-open set} is a domain of slice regularity.
\end{prop}

{
\begin{proof}
	Suppose that an axially symmetric slice-open set $\Omega$ is not a domain of slice regularity. By Proposition \ref{pr-so}, there is an $I$-triple $(\Omega,U_1,U_2)$ for some $I\in\mathbb{S}$. Using the fact that $\Omega$ is axially symmetric, and Theorem \ref{pr-ef} where we set $\mathbb{I}=(I,-I)$, $\mathbb{U}=(\Omega_I,\Omega_I)$, we deduce that any holomorphic function $f:\Omega_I\rightarrow\mathbb{C}_I$ can be extended to a slice regular function $\widetilde f$ defined on $\Omega$. Since $(\Omega,U_1,U_2)$ is an $I$-triple, $f|_{U_1}=\widetilde f|_{U_1}$ can be extended to a holomorphic function $\breve f:U_2\rightarrow\mathbb{H}$. By Splitting Lemma \ref{lm-sl} and Identity Principle in complex analysis, $\breve f$ is $\mathbb{C}_I$-valued holomorphic function.
	Thus, for any holomorphic function $f:\Omega_I\rightarrow\mathbb{C}_I$, there is a function $\breve f:U_2\rightarrow\mathbb{C}_I$ such that $f=\breve f$ on $U_1$. We conclude that $\Omega_I$ is not a domain of holomorphy in $\mathbb{C}_I$, which is a contradiction.
\end{proof}
}

\begin{prop}
	Any $\sigma$-ball is a domain of slice regularity.
\end{prop}

\begin{proof}
	Let $p\in\mathbb{H}$ and $r\in\mathbb{R}_+$, let $\Sigma(p,r)$ be the $\sigma$-ball with center at $p$ and with radius $r$. Consider the function $f:\Sigma(p,r)\rightarrow\mathbb{H}$ defined by
	\begin{equation*}
		f(q)=\sum_{n\in\mathbb{N}}\left(\frac{q-p}{r}\right)^{*2^n}.
	\end{equation*}

{ We know that $p\in\mathbb{C}_K$, for some $K\in\mathbb{S}$, and from classical complex analysis arguments, we have that $f_K:\, \Sigma_K(p,r)\to\mathbb{C}_K$
does not extend to a holomorphic function near any point of the boundary of $\Sigma_K(p,r):=\Sigma(p,r)\cap\mathbb{C}_K$.	If $\Sigma(p,r)$ is not a domain of slice regularity, then by Proposition \ref{pr-so}, there is an $I$-triple $(\Sigma(p,r),U_1,U_2)$ for some $I\in\mathbb{S}$. Let $U_1'$ be a connected component of $\Sigma(p,r)\cap U_2$ in $\mathbb{C}_I$ with $U_1'\cap U_1\neq\varnothing$. Then $(\Sigma(p,r),U_1',U_2)$ is also an $I$-triple and $U_2\cap\partial_I U_1'\subset\partial_I(\Sigma_I(p,r))$. If $p\in\mathbb{C}_I$, the holomorphic function $f_I:\Sigma_I(p,r)\rightarrow\mathbb{C}_I$ can be extended to a holomorphic function near a point of the boundary of $\Sigma_I(p,r)$, which is a contradiction.
	
	Otherwise, if $p\notin\mathbb{C}_I$, then $p\in\mathbb{C}_J^+$ for some $J\in\mathbb{S}\backslash\{\pm I\}$. Take $z=x+yL\in U_2\cap\partial_I U_1'$ with $y>0$ and $L\in\{\pm I\}$. Then $x+yJ\in\Sigma_J(p,r)$ and $x-yJ\in\partial_J(\Sigma_J(p,r))$. There is $r_1\in\mathbb{R}_+$ such that
	\begin{equation*}
		B_L(x+yL,r_1)\subset U_2\qquad\mbox{and}\qquad B_J(x+yJ,r_1)\subset\Sigma(p,r)\cap\mathbb{C}_J^+.
	\end{equation*}
	Using Theorem \ref{pr-ef} where we set
	\begin{equation*}
		\mathbb{I}=(L,J)\qquad\mbox{and}\qquad \mathbb{U}:=(\Sigma_L(p,r)\cup B_L(x+yL,r_1),\Sigma_J(p,r)),
	\end{equation*}
	 we conclude that the holomorphic function $f_J:\Sigma_J(p,r)\rightarrow\mathbb{C}_J$ can be extended to a holomorphic function on $B_J(x-yJ,r_1)\cup\Sigma_J(p,r)$ near the point $x-yJ\in\partial_J(\Sigma_J(p,r))$, which is a contradiction. Therefore, $\Sigma(p,r)$ is a domain of slice regularity.
	}
\end{proof}

\begin{prop}\label{pr-dsr}
	Let $I\in\mathbb{S}$ and $\Omega$ is a domain of slice regularity. If $\gamma\in\mathscr{P}(\mathbb{C})$ and $(J,K)\in\mathbb{S}^2_*$ with $\gamma^J,\gamma^K\subset\Omega$, then $\gamma^I\subset\Omega$ for any $I\in\mathbb{S}$.
\end{prop}

\begin{proof}
	We shall prove this  by contradiction. Suppose that $\gamma^I\not\subset \Omega$ for some $I\in\mathbb{S}$. Since $\gamma^I$ is a slice-path, $(\gamma^I)^{-1}(\Omega)$ is open in $[0,1]$. Set $$t:=\min\{s\in[0,1]:\gamma^I(s)\notin\Omega\}.$$
	
By assumption, we have  $$z_J:=\gamma^J(t)\in \Omega, \qquad z_K:=\gamma^K(t)\in\Omega$$  so that  $$B_J(z_J,r)\subset\Omega, \qquad B_K(z_K,r)\subset\Omega$$ for some $r\in\mathbb{R}_+$.

Since $\gamma^I$ is continuous in $\mathbb{C}_I$, there is $t'\in[0,t)$ such that  $\gamma^I(t')\in B_I(z_I,r)$, where $z_I:=\gamma^I(t)$.
	
	For any slice regular function $f$ on $\Omega$, define a function $g:B_I(z_I,r)\rightarrow\mathbb{H}$ by
	\begin{equation}\label{eq-ij}
	g(x+yI)=(I-K)(J-K)^{-1}f(x+yJ)+(I-J)(K-J)^{-1}f(x+yK)
	\end{equation}
	for any $x,y\in\mathbb{R}$ with $x+yJ\in B_J(z_J,r)$.

By direct calculation (see the proof of \cite[Theorem 3.2]{Colombo2009001}), $g$ is holomorphic. Note that $\gamma^I(t')\in\Omega_I\cap B_I(z_I,r)$. Then by taking $J=I$ and $K=-I$ in \eqref{eq-ij}, we have
\begin{equation*}
	g(x+yI)=f(x+yI),\qquad\forall\ x+yI\in\Omega_I\cap B_I(z_I,r).
\end{equation*}
Hence $g=f$ near $\gamma^I(t')$.

By Corollary  \ref{ex-sb}, there is a  unique slice regular extension $\widetilde{g}$ on $\Omega_1:=\Sigma(z_I,r)$. Since $\Sigma(z_I,r)$ and $\Omega$ are slice-open, it follows that $\Sigma(z_I,r)\cap\Omega$ is slice-open. Hence the slice-connected component $\Omega_2$ of $\Sigma(z_I,r)\cap\Omega$ containing $\gamma^I(t')$ is an st-domain. By Identity Principle \ref{th-ip}, $f=g$ on $\Omega_2$.
	
	It is easy to check that $\Omega$, $\Omega_1$ and $\Omega_2$ satisfy (i-iii) in Definition \ref{df-dsr}. Hence $\Omega$ is not a domain of slice regularity, which is a contradiction.
\end{proof}

\section{Final remarks}

	Let $\widetilde{\Omega_\varphi}$ and $\widetilde{\Psi_\varphi}$ be defined as in Proposition \ref{pr-sre}. Proposition \ref{pr-fsd} implies that $\widetilde{\Psi_\varphi}$ cannot be slice regularly extended to a larger st-domain. However, according to Proposition \ref{pr-dsr}, $\widetilde{\Omega_\varphi}$ is not a domain of slice regularity when $\varphi$ is not constant. This suggests  to establish an analogue of the theory of Riemann domains for quaternions and characterize   the domain of existence of $\widetilde{\Psi_\varphi}$ which is an analogue of a Riemann domain. Since the slice topology is not Euclidean near $\mathbb{R}$,  we cannot consider quaternionic manifolds along the lines used in this paper. Instead, orbifolds over $(\mathbb{H},\tau_s)$ could be considered.

{\bf Acknowledgments}. The authors are grateful to the referees for carefully reading of the manuscript and for the useful comments.

\bibliographystyle{amsplain}

\end{document}